\documentclass[onefignum,onetabnum]{siamart220329}
\newsiamremark{remark}{Remark}
\usepackage[utf8]{inputenc}
\usepackage{graphicx} % Required for inserting images
\usepackage{comment}
\usepackage{amsmath, amssymb, amscd}
\usepackage{graphicx}
\usepackage{xcolor}
\usepackage{multirow}
\usepackage{subcaption}
\usepackage{bm}
\usepackage{algorithm}
\usepackage{algpseudocode}
\usepackage{multirow}
\usepackage{mathtools}
\usepackage{comment}
\usepackage{xspace}

\newcommand{\bb}[1]{\mathbb{#1}}
\newcommand{\mbf}[1]{\mathbf{#1}}

\newcommand{\mc}[1]{\mathcal{#1}}
\newcommand*{\bmat}[1]{\begin{bmatrix}#1\end{bmatrix}}

\newcommand{\orth}{\textup{\texttt{orth}}}
\newcommand{\MATLAB}{\textsc{Matlab}\xspace}

\DeclareMathOperator{\rank}{rank}
\DeclareMathOperator{\Range}{Range}

\DeclareMathOperator{\myspan}{span}
\DeclareMathOperator*{\argmin}{argmin}

\usepackage{cite}
\usepackage{todonotes}

\title{Time-Domain Iterative Rational Krylov Method}
\author{Michael S. Ackermann\thanks{Department of Mathematics, Virginia Tech, Blacksburg , VA, 24061  (\email{amike98@vt.edu}).}
\and Serkan Gugercin \thanks{Department of Mathematics and Division of Computational Modeling and
    Data Analytics, Academy of Data Science, Virginia Tech, Blacksburg, VA 24061} (\email{gugercin@vt.edu}).}

\date{March 2023}

\begin{document}

\maketitle
\begin{abstract}
    The Realization Independent Iterative Rational Krylov Algorithm (TF-IRKA) is a frequency-based data-driven reduced order modeling (DDROM) method that constructs $\mathcal H_2$ optimal DDROMs. However, as the $\mathcal H_2$ optimal approximation theory dictates, 
   TF-IRKA requires repeated sampling of
  frequency data, that is, values of the system transfer function and its derivative, outside the unit circle. This repeated evaluation of frequency data requires repeated full model computations and may not be feasible. The data-informativity framework for moment matching provides a method for obtaining such frequency data from a single time-domain simulation. However, this framework usually requires solving linear systems with prohibitively ill-conditioned matrices, especially when recovering frequency data from off the unit circle as required for optimality.  
    In this paper, building upon our previous work with the data informativity framework for moment matching, we provide
    a formula for the nonzero extreme eigenvalues of a symmetric rank-$1$ perturbation to an orthogonal projection, which then leads to an optimal scaling of the aforementioned linear systems.  We also establish connections between the underlying dynamical system and conditioning of these linear systems.  This analysis then leads to our algorithmic development, time-domain IRKA, which allows us to implement a time-domain variant of TF-IRKA, constructing $\mathcal H_2$ optimal DDROMs from a single time-domain simulation without requiring repeated frequency data evaluations. The numerical examples illustrate the effectiveness of the proposed algorithm. 
\end{abstract}


\section{Introduction}
High-fidelity modeling of complex systems often results in large scale dynamical system models.  The time, energy, or storage cost of simulating these models may make them infeasable for their intended applications.  When this is the case, reduced order models (ROMs) are constructed and used in place of the full order model (FOM).  The goal of reduced order modeling is to find a ROM that well well approximates the FOM while requiring a fraction of the resources to simulate.

Classical model reduction techniques  take an intrusive approach where one assumes to have access to a realization of the FOM that fully represents the dynamics of the system.  These approaches are broadly called ``projection-based methods" as they project the dynamics of the FOM to a smaller subspace that captures most of the energy of the system, see e.g.,  \cite{Antoulas2005ApproxDynamSys,AntoulasBG2020Book,BennerGW2015SurveyProjMethods,DrohmannHO2012ReducedBasis,Moore1981PrincipleComponentAnalysis,QuarteroniMN2015RedBasisBook,RozzaHP2008ReducedBasis}. 
While these approaches have had great success in producing highly accurate ROMs, they are limited in applicability by assuming access to a realization of the FOM.  In some scenarios, such a realization may be unavailable, motivating the need for data-driven approaches.

Data-driven reduced order modeling does not assume access to a realization of the FOM, but rather assumes input-output data in the time or frequency domain is available, see, e.g., \cite{AntoulasBG2020Book,BerlijafaG2017RKFIT,conniAAARationalApproximation2024,DramacGB2015QuadVF,GustavsenS1998VFOrig,GrivetTalocia2004TimeDomainVF,hokansonM2018RationalLS,MayoA2007LoewnerFramework,NakatsukasaST2018AAA,semlyenG1999VF,MliG23,GosWG24}.  Many of the classical projection-based techniques such as optimal interpolatory projections and balanced truncation have been reformulated in this new input-output data-driven setting, see, e.g., \cite{BeattieG2012TFIRKA,GoseaGB2021QuadBT,MayoA2007LoewnerFramework}. We note that there are also data-driven methods that assume access to internal state measurements in addition to input-output measurements we consider here, see e.g., \cite{bruntonK2019DataDrivenSEbook, Hokanson2017ProjNonlinLS,KutzBBP2016DMDBook, Moore1981PrincipleComponentAnalysis,QianPW2020Lift&Learn,RajendraB2020DLROM,Schmid2022DMD}.  In our work we do not assume to know the internal state variable, therefore these methods are not considered here.

In the present work, we focus on linear, time invariant (LTI) systems. These systems typically arise from discretization of partial differential equations (PDEs). They can occur in the form of continuous-time LTI systems (semi-discretized PDEs) or in the form of discrete-time LTI systems (fully-discretized).  In this work, we consider discrete-time LTI systems, but remark that any continuous time LTI system can be transformed to a discrete-time system via e.g., a backwards stable interactive method; see \cite{Antoulas2005ApproxDynamSys} for details.

Both continuous and discrete LTI systems can be studied in either the time-domain or frequency-domain.  One may switch between these equivalent descriptions via the Laplace transform if the system is continuous or the Z-transform if the system is discrete. The mapping from input to output in the frequency domain is called the transfer function of the underlying system and plays a fundamental role in constructing reduced models.

Frequency based reduced order modeling techniques have had great success in modeling LTI systems.  Indeed, optimal approximation of LTI systems in the $\mc H_2$ norm requires interpolation in the frequency domain.  If one has access to a realization of the FOM (i.e., access to full-order operators), one method to compute a locally $\mc H_2$ optimal ROM is by the Iterative Rational Krylov Algorithm (IRKA) \cite{GugercinAB2008H2}.  If, on the other hand, one has only access to samples of the transfer function and its derivative, Realization Independent (TF)-IRKA \cite{BeattieG2012TFIRKA}, a data-driven formulation of IRKA, may be used.    
However, in some scenarios, neither a realization of the FOM nor direct measurements of the transfer function are available. Instead, one may be given access to only time-domain output data for a given input.
The purpose of this paper is to extend TF-IRKA algorithm to be compatible with time-domain data so that one may construct  $\mc H_2$-optimal reduced models using only a single time-domain trajectory. 

The first step of such a task is to develop a method to infer frequency information from available time-domain information.  One such method was proposed in \cite{burohmanBSC2020informativity}, where the authors provided the theoretical conditions that time-domain data must satisfy for one to infer the value of the FOM transfer function at some $\sigma \in \bb C$.  Despite its strong theoretical foundation, applying this framework in practical scenarios required careful analysis of the underlying numerical schemes. Our recent work~\cite{AckermannG2023FreqROMTD} 
has provided such an analysis and extended upon the contributions of  \cite{burohmanBSC2020informativity}, leading to a robust numerical framework for calculating values of the frequency response of the FOM from a single time-domain simulation.

Despite the contributions of \cite{AckermannG2023FreqROMTD}, extending TF-IRKA to work with time-domain data raises many theoretical and numerical complications.  In particular, TF-IRKA requires recovering values and derivatives of the transfer function at $\sigma \in \bb C$ where $|\sigma| > 1$, which introduces additional ill-conditioning in linear systems to be solved.  This paper addresses these complications, building on the work of \cite{AckermannG2023FreqROMTD} to provide an improved method for calculating transfer function values from time-domain data, as well as the Time-Domain (TD)-IRKA algorithm, an algorithm that constructs $\mc H_2$ optimal ROMs directly from time-domain data.  The main contributions of this paper are as follows:
\begin{itemize}
    \item \Cref{sec:explicitCondNumForm} provides a formula for the condition number of any sub-unitary matrix appended by an additional column $\bmat{\mbf Q &\mbf z}$ by finding an explicit formula for the eigenvalues of $\mbf Q\mbf Q^* + \mbf z \mbf z^*$.  This formula is then used to determine the optimal preconditioner of a certain class for linear systems involved in calculating frequency information from time-domain data
    \item \Cref{sec:CondNumFromSysProps} proves a relationship between the underlying FOM and the condition number of the linear system used to recover frequency information from the given time-domain data.  This formula is further analyzed to reveal more methods for reducing the linear system condition number.
    \item \Cref{sec:TD-IRKA} develops the Time-Domain (TD)-IRKA algorithm by incorporating the results of \cref{sec:Conditioning} and providing additional implementation details.
    \item \Cref{sec:results} demonstrates the effectiveness of TD-IRKA on several benchmark examples.
\end{itemize}
We begin with \cref{sec:preliminaries}, which introduces several necessary preliminaries.
\section{Preliminaries}
\label{sec:preliminaries}
Let $\mc S$ be an order $n$ discrete-time, asymptotically stable, single-input-single-output (SISO) dynamical system with the state-space form
\begin{equation}
    \label{eq:LinDiscSys}
    \mathcal S: \left \{
    \begin{aligned}
    \mbf E \mbf x[k+1] &= \mbf A\mbf x[k] + \mbf bu[k]\\
    y[k+1] &= \mbf c^{\top}\mbf x[k],\\
    \end{aligned} \right.
\end{equation}
where $\mbf E \in \bb R^{n\times n}$ is nonsingular, $\mbf A \in \bb R^{n \times n}, \,\mbf b \in \bb R^n,\, \mbf c^{\top} \in \bb R^{1\times n}$;
$\mbf x[k] \in \bb R^n$ is the state of \cref{eq:LinDiscSys} at time $k$; $u[k] \in \bb R$ is the input at time $k$; and $y[k]\in \bb R$ is the  output at time  $k$.
By taking the $Z$-transform of dynamics in~\cref{eq:LinDiscSys}, one obtains the transfer function of $\mathcal S$:
\begin{equation}
    \label{eq:TranferFuncDef}
    H(z) = \mbf c^{\top}(z\mbf E-\mbf A)^{-1}\mbf b,
\end{equation}
which is an order $n$ rational function in $z$. 
The asymptotic stability of $\mathcal S$ means that the 
eigenvalues of the matrix pencil $\lambda \mbf E - \mbf A$, i.e., the poles of $H(z)$, lie in the open unit disk. For $r < n$, an order $r$ reduced order model of \cref{eq:LinDiscSys} is given by
\begin{equation}
    \label{eq:LinDiscSys_rom}
    \widehat{\mathcal S}: \left \{
    \begin{aligned}
    \mbf E_r \mbf x_r[k+1] &= \mbf A_r\mbf x_r[k] + \mbf b_ru[k]\\
    y_r[k+1] &= \mbf c_r^{\top}\mbf x_r[k],\\
    \end{aligned} \right.
\end{equation}
where $\mbf E_r \in \bb R^{r\times r}$ is nonsingular, $\mbf A_r \in \bb R^{r \times r}, \,\mbf b_r \in \bb R^r,\, \mbf c_r^{\top} \in \bb R^{1\times r}$;
$\mbf x_r[k] \in \bb R^r$ is the \emph{reduced} state of \cref{eq:LinDiscSys_rom} at time $k$; $u[k] \in \bb R$ is the input at time $k$; and $y_r[k]\in \bb R$ is the \emph{reduced} output at time $k$.  The transfer function of \cref{eq:LinDiscSys_rom} is given by the order $r$ rational function
\begin{equation}
    \label{eq:ROMTF}
    \widehat H(z) = \mbf c_r^{\top}(z\mbf E_r-\mbf A_r)^{-1}\mbf b_r.
\end{equation}

The rest of this section summarizes various mathematical preliminaries required to construct $\mc H_2$ optimal (defined in the next section) DDROMs of a system as in \cref{eq:LinDiscSys} from samples of the time-domain input and output:
{\small \begin{equation}
\label{eq:UY} \bb U = \begin{bmatrix}u[0] & u[1] & \ldots & u[T]\end{bmatrix}^{\top} \in \mathbb R^{T+1}  ~~\mbox{and}~~ \bb Y = \begin{bmatrix}y[0] & y[1] & \ldots & y[T]\end{bmatrix}^{\top} \in \mathbb R^{T+1},\end{equation}}

\noindent
respectively.  Specifically, we assume that the quantities $\mbf E, \mbf A,\mbf b,$ and $\mbf c$ of \cref{eq:LinDiscSys} and \cref{eq:TranferFuncDef} are unknown, but we have the freedom to choose an input $\bb U$ and observe the response of $\mc S$ when driven by $\bb U$, denoted by $\bb Y$.

\subsection{$\mathcal H_2$ optimality conditions}
\label{sec:H2OptConds}
For an order $n$ system $\mc S$ as in \cref{eq:LinDiscSys} with the associated transfer function $H(z)$ as in \cref{eq:TranferFuncDef}, its 
$\mathcal H_2$ norm is norm is defined as
\begin{equation}
    \|H\|_{\mathcal H_2} = \sqrt{\frac{1}{2\pi}\int_{-\pi}^{\pi} |H(e^{\mbf i \omega})|^2 d\omega},
\end{equation}
which is the $\mathcal L_2$ norm of the transfer function on the unit circle. The $\mathcal H_2$ norm provides an upper bound for the $\mathcal L_\infty$ norm of the output $y(t)$; in other words
\begin{equation} \label{eq:h2bound}
  \| y\|_{{\mc L}_{\infty}}  \leq  \|H\|_{\mathcal H_2} 
  \| u\|_{{\mc L}_{2}}.
\end{equation}
Then, given $\mc S$ with transfer function $H(z)$,
 the $\mathcal H_2$ optimal reduced order modeling problem is to find a degree $r < n$ 
 model $\widehat{\mc S}$ with transfer function $\widehat H(z)$ that satisfies
\begin{equation}
    \label{eq:H2OptProblem}
    \widehat H = \argmin_{\widetilde{H}} \|H - \widetilde{H}\|_{\mathcal H_2},
\end{equation}
with $\widetilde H$ ranging over all degree $r$ rational functions with poles in the open unit disc. The output norm inequality \cref{eq:h2bound} provides the motivation for the $\mathcal H_2$ optimal reduced order modeling since~\cref{eq:h2bound} implies 
\begin{equation} \label{eq:h2Errbound}
  \| y - y_r\|_{{\mc L}_{\infty}}  \leq  \|H - \widehat H \|_{\mathcal H_2} 
  \| u\|_{{\mc L}_{2}},
\end{equation}
where $y_r$ is the output of $\mc S_r$.  
Therefore, minimizing $\| H - \widehat H \|_{{\mc H}_{2}}$ guarantees that $y_r$ is a high-fidelity approximation to $y$.

Minimizing \cref{eq:H2OptProblem} is a non-convex optimization problem, so the typically
approach is to seek locally optimal solutions.
There are equivalent characterizations of local $\mathcal H_2$ optimality via Sylvester equations and interpolation conditions \cite{GugercinAB2008H2,HylandB1985OptimalProj,MeierL1967InterpOptConds,wilson1970LyapunovConds}. 
The Sylvester equation framework requires access to the original state-space model. 
As we seek to construct DDROMs from data, we will utilize the interpolation based  (data-driven) optimality criteria: 
\begin{theorem}
    \label{thm:H2optCond}
    Let $\widehat H$ be a locally optimal order $r$ rational approximation to $H$ in the $\mathcal H_2$ norm.  Let $\{\lambda_i\}_{i = 1}^r$ be the poles of $\widehat H$.  Then $\widehat H$ is a Hermite interpolant to $H$ at $\frac{1}{\lambda_i}$.  That is, for each $i = 1,\ldots,r$,
    \begin{equation}
        \label{eq:H2optcond}
        \begin{aligned}
            \widehat H\left(\frac{1}{\lambda_i}\right) &= H\left(\frac{1}{\lambda_i}\right)\quad \mbox{and} \quad
            \widehat H'\left(\frac{1}{\lambda_i}\right) &= H'\left(\frac{1}{\lambda_i}\right),
        \end{aligned}
    \end{equation}
    where $H'$ and $\widehat H'$ denote the derivatives of $H$ and $\widehat H$ with respect to $z$.
\end{theorem}

For a proof of \cref{thm:H2optCond} and its relation to other equivalent optimality conditions, see \cite{AntoulasBG2020Book,GugercinAB2008H2}. These optimality conditions do not require $H(z)$ to be a rational function as in \eqref{eq:TranferFuncDef}; it just needs to be a $\mathcal{H}_2$ function.  Only $\widehat{H}(z)$ is required to be a rational function. 
Note that in this interpolatory setting 
``data'' correspond to evaluations of $H(z)$ and $H'(z)$ at some frequencies $z=\sigma_i$.
Therefore,  
to construct a data-driven locally optimal reduced model $\widehat H$, a method to construct rational Hermite interpolants to $H$ is required
given the evaluations of $H(\sigma_i)$ and $H'(\sigma_i)$.  Further, the optimal interpolation points $\sigma_i$'s are the reciprocal of the \emph{reduced} system poles, which are not known a priori. Thus, an iterative algorithm to converge to the optimal model from some initial set of interpolation points is needed.  We first show how to construct Hermite interpolants from data in \cref{sec:HerLowInterp}, and then summarize the transfer function iterative rational Krylov algorithm (TF-IRKA) \cite{BeattieG2012TFIRKA}, an algorithm to find an optimal reduced model from data samples $H(\sigma_i)$ and $H'(\sigma_i)$ in \cref{sec:TF-IRKA}.

\subsection{Data-driven Hermite interpolation}
\label{sec:HerLowInterp}
Given $r$ distinct points $\{\sigma_i\}_{i=1}^r \subset \bb C$ and corresponding values $\{H(\sigma_i)\}_{i=1}^r$ and derivatives $\{H'(\sigma_i)\}_{i=1}^r$, there exists an (at most) order $r$ rational function $H_r$ that interpolates the data, i.e.,
\begin{equation}
    \label{eq:hermcond}
    \widehat H(\sigma_i) = H(\sigma_i)\quad \text{and}\quad \widehat H'(\sigma_i) = H'(\sigma_i)\quad \text{for } i=1,2,\ldots,r.\end{equation}
This Hermite interpolant may be found via the \emph{Hermite Loewner framework} \cite{MayoA2007LoewnerFramework}, summarized below.

Given the data $\{H(\sigma_i)\}_{i=1}^r$ and $\{H'(\sigma_i)\}_{i=1}^r$ sampled at distinct $\sigma_i$, first form the Loewner matrix $\mbf L \in \bb C^{r \times r}$ and shifted Loewner matrix $\mbf M \in \bb C^{r \times r}$ as:
\begin{subequations}
    \parbox[c]{\textwidth*13/30}{
    \begin{equation}
    \label{eq:HermiteLowL}
    \mbf L_{ij} = 
    \begin{cases} 
      -\frac{H(\sigma_i) - H(\sigma_j)}{\sigma_i-\sigma_j}, & \text{if } i \neq j \\
      -H'(\sigma_i), & \text{if } i = j \\
   \end{cases}
   \end{equation}
   }
   \parbox[c]{\textwidth*15/30}{
    \begin{equation}
    \label{eq:HermiteLowM}
    \mbf M_{ij} = 
    \begin{cases} 
      -\frac{\sigma_iH(\sigma_i) - \sigma_jH(\sigma_j)}{\sigma_i-\sigma_j}, & \text{if } i \neq j \\
      -(H(\sigma_i) + \sigma_iH'(\sigma_i)), & \text{if } i = j. \\
   \end{cases}
   \end{equation}
   }
\end{subequations}

 \noindent
Let $\mbf q = \bmat{H(\sigma_1) & \ldots & H(\sigma_r)}$.  If the matrix pencil $z\mbf L-\mbf M$ is nonsingular for each $z = \sigma_i$, $i = 1, \ldots, r$, then the degree-$r$ transfer function $\widehat H(z) = \mbf q^{\top}(z\mbf L-\mbf M)^{-1}\mbf q$ interpolates the data. 
In other words, $\widehat H(z) =
\mbf c_r^{\top}(z\mbf E_r -\mbf A_r)^{-1}\mbf b_r$ with $\mbf E_r = \mbf L$, $\mbf A_r = \mbf M$, $\mbf b_r = \mbf c_r = \mbf q$ 
satisfies~\cref{eq:hermcond}.
In general, if the data samples and interpolation points are complex, $\mbf L,\mbf M$,  and $\mbf q$ will have complex entries.  In the case where the interpolation points $\sigma_i$ are closed under conjugation and the underlying system is real (as will be the case in this paper), a state transformation can be performed to obtain a real realization of the reduced model, see, e.g., \cite{AntoulasBG2020Book} for details.

\subsection{The Transfer Function Iterative Rational Krylov Algorithm}
\label{sec:TF-IRKA}
\cref{thm:H2optCond} provides a means to check if a given order $r$ rational function $\widehat H$ is an $\mathcal H_2$ optimal approximation to $H$, but does not provide a method for finding such a reduced model.  The Iterative Rational Krylov Algorithm (IRKA) \cite{GugercinAB2008H2} is a fixed point iteration that constructs a locally $\mathcal H_2$ optimal approximation to a given first-order linear dynamical system by iteratively updating interpolatory projections.  IRKA is well suited to large-scale problems and often converges rapidly to a local solution of \cref{eq:H2OptProblem} (see \cite{FlaggBG2012IRKAConverge}), but is restricted to first-order linear dynamical systems and requires an explicit realization of the system, that is, the matrices $\mbf E, \mbf A, \mbf b$, and $\mbf c$ in \cref{eq:TranferFuncDef} are assumed to be given.

The Realization Independent Iterative Rational Krylov Algorithm (TF-IRKA), \cref{alg:IRKA} extends IRKA to the case where an explicit realization of a system is unavailable, but transfer function value and derivative samples are available.  Thus, TF-IRKA is a data-driven approach as it does not require a realization of the system.  TF-IRKA also only assumes that the underlying transfer function is $\mathcal H_2$, not necessarily rational.
We provide an algorithmic summary of TF-IRKA in \cref{alg:IRKA}.  Upon convergence, \cref{alg:IRKA} returns matrices that define a locally $\mc H_2$ optimal reduced model to the full order model that produces the data samples.

\begin{algorithm}[!htb]
\caption{TF-IRKA}\label{alg:IRKA}
\begin{algorithmic}
\Require $\{\sigma_i\}_{i = 1}^r$, an initial set of interpolation points closed under conjugation (i.e., if $\sigma$ is an interpolation point, so is $\overline{\sigma}$).
\State Sample $H$ and $H'$ at $\{\sigma_i\}_{i = 1}^r$
\State Form $\mbf E_{r}$ and $\mbf A_{r}$ using $\sigma_i$, $H(\sigma_i)$, and $H'(\sigma_i)$ via Hermite Loewner interpolation.
\While{Not Converged}
\State Calculate generalized eigenvalues $\{\lambda_i\}_{i=1}^r$ of the pencil $\lambda\mbf E_{r} - \mbf A_{r}$
\State $\sigma_i \gets \frac{1}{\lambda_i}$
\State Resample $H$ and $H'$ at new $\{\sigma_i\}_{i = 1}^r$
\State Form new $\mbf E_{r} = \mbf L$ as in \cref{eq:HermiteLowL} and $\mbf A_{r} = \mbf M$ as in \cref{eq:HermiteLowM}  with $\sigma_i$, $H(\sigma_i)$, and $H'(\sigma_i)$.
\EndWhile
\State Form final $\mbf E_{r},\mbf A_{r},\mbf b_{r},\mbf c_{r}$ via Hermite Loewner interpolation (\cref{sec:HerLowInterp})
\end{algorithmic}
\end{algorithm}

TF-IRKA (\cref{alg:IRKA}) in its core is an 
iteratively corrected data-driven rational Hermite interpolant where interpolation points are updated in every step based on the optimality conditions~\eqref{eq:H2optcond}.  
Therefore, TF-IRKA
requires resampling $H$ and $H'$ at new interpolation points $\sigma_i$ at every step.  This may be challenging in practice because resampling may require extensive computations or another run of a physical experiment.  Additionally, since the poles of the reduced model should lie inside the unit circle, the optimal interpolation points, which are reciprocal of the poles, will lie outside the unit circle. Thus
\cref{alg:IRKA} and \cref{thm:H2optCond} require sampling of $H$ and $H'$ at $\sigma$ where $|\sigma| > 1$.  This data is difficult to obtain from physical set ups with currently available methods.  The next section summarizes a method to obtain $H(\sigma)$ and $H'(\sigma)$ for any $\sigma \in \bb C$ from a single time-domain simulation of the system.

\subsection{Recovering frequency information from time-domain data}
Frequency-based data-driven reduced order modeling methods require access to values and derivatives of the transfer function $H(z)$ given in~\cref{eq:TranferFuncDef}.  When one only has access to time-domain input-output data $(\bb U,\bb Y)$ as in~\eqref{eq:UY}, a method is needed to infer frequency information ($H(\sigma)$ and $H'(\sigma)$) from the available time-domain data.  In \cite{burohmanBSC2020informativity}, the authors presented a new method based on the data informativity framework \cite{vanWaardeECT2023DataInfSurvey,VanWaardeETC2020DataInform} to recover the values and derivatives of $H(z)$ from time-domain data at some specified $\sigma \in \bb C$.  In \cite{AckermannG2023FreqROMTD}, we provided an analysis of the method of \cite{burohmanBSC2020informativity}, which led to improved conditioning of associated linear systems
and an error indicator, among other contributions.  We briefly summarize some necessary results from those works here; for details we refer the reader to the original works~\cite{AckermannG2023FreqROMTD,burohmanBSC2020informativity}.

To simplify the notation, we define the Hankel matrix of depth $n$ from the input data $\bb U \in \bb R^{T+1}$ using the notation 
\begin{equation}
    \label{eq:hankelMatDef}
    \bb H_n(\bb U) = 
    \begin{bmatrix}
        u[0]&u[1]&\ldots &u[{T-n}]\\
        u[1]&u[2]&\ldots&u[{T-n+1}]\\
        \vdots&\vdots &\ddots & \vdots \\
        u[n] & u[{n+1}] & \ldots & u[T]
    \end{bmatrix} 
    \in \bb R^{(n+1) \times (T-n + 1)}.
\end{equation}
The Hankel matrix of depth $n$ from the input data $\bb Y \in \bb R^{T+1}$, $\bb H_n(\bb Y)$, is defined similarly.  We also define the vectors
\begin{center}
  \begin{subequations}
    \begin{equation}
        \label{eq:gamma_def}
        \gamma_n(\sigma) = \begin{bmatrix}1 & \sigma & \ldots & \sigma^n\end{bmatrix}^{\top}  \in \bb C^{n+1}~~\mbox{and}
    \end{equation}
    \begin{equation}
        \label{eq:gamma1_def}
        \gamma^{(1)}_n(\sigma) = \begin{bmatrix}0 & 1 & 2\sigma & \ldots & n\sigma^{n-1}\end{bmatrix}^{\top} \in \bb C^{n+1}.
    \end{equation}
    \end{subequations}
\end{center}
Finally, we define  $\mbf G_n \in \bb R^{2(n+1)\times(T-n+1)}$, and
$\mbf z(\sigma)$, $\mbf b(\sigma)$, $\mbf b^{(1)}(\sigma) \in \bb R^{2(n+1)}$:  
\begin{center}
  \begin{subequations}
    \parbox[c]{\textwidth*14/30}{
    \begin{equation}
        \label{eq:GnDef}
        \mbf G_n \coloneqq \bmat{\bb H_n(\bb U)\\ \bb H_n(\bb Y)} 
    \end{equation}
    }
    \parbox[c]{\textwidth*14/30}{
    \begin{equation}
        \label{eq:zSigmaDef}
        \mbf z(\sigma) \coloneqq \bmat{\mbf 0 \\ -\gamma_n(\sigma)} 
    \end{equation}
    }
    
    \parbox[c]{\textwidth*14/30}{
    \begin{equation}
        \label{eq:bSigmaDef}
        \mbf b(\sigma) \coloneqq \bmat{\gamma_n(\sigma) \\ \mbf 0} 
    \end{equation}
    }
    \parbox[c]{\textwidth*14/30}{
    \begin{equation}
        \label{eq:bSigmaPrimeDef}
        \mbf b^{(1)}(\sigma) \coloneqq \bmat{\gamma^{(1)}_n(\sigma) \\ M_0\gamma_n^{(1)}(\sigma)},
    \end{equation}
    }
\end{subequations}  
\end{center}
where $\mbf 0$ is the zero vector of length $n+1$.  
The method to recover frequency information $H(\sigma)$ and $H'(\sigma)$ from time domain information $\bb U$ and $\bb Y$ presented in 
\cite{AckermannG2023FreqROMTD} (which is equivalent to the method of \cite{burohmanBSC2020informativity} in exact arithmetic)
is given in \cref{prop:ExistUniqueOrth}, which summarizes the results of Corrollaries 3.2 and 3.3 of \cite{AckermannG2023FreqROMTD}.
\begin{remark}
    Since this work is concerned with the computational recovery of quantities that have an exact value and these computations are assumed to be carried out on a machine with finite precision arithmetic, we find it necessary to distinguish between the true quantities and their numerically calculated values.  We denote the true value of a transfer function $H$ at $\sigma \in \bb C$ as $H(\sigma)$ and the calculated (inferred) value as $M_0$.  Similarly, we use $H'(\sigma)$ and $M_1$ for the derivatives.  We emphasize that in the results presented below, if computations are carried out in exact arithmetic, then $M_0 = H(\sigma)$ and $M_1 = H'(\sigma)$.  
\end{remark}
\begin{lemma}
    \label{prop:ExistUniqueOrth}
    For a given LTI system $\mc S$ as in~\eqref{eq:LinDiscSys},
    assume we have access to the input data $\bb U \in \bb R^{T+1}$ and its corresponding output data $\bb Y \in \bb R^{T+1}$ as in~\eqref{eq:UY}.  Let $n$ be the order of the system and let $\sigma, M_0 \in \mathbb C$.
    Let $\mbf U \in \bb R^{2(n+1)\times p}$
    be an orthonormal basis for the range of $\mbf G_{n}$ \cref{eq:GnDef}, where $p \leq T-n+1$.
    Then the data $(\bb U$,$\bb Y)$ are informative for interpolation at $\sigma$ (i.e., we may recover $M_0 = H(\sigma)$ from the data) if and only if there exists $\xi \in \bb C^{p}$ such that
    \begin{equation}
        \label{eq:CalcM0Orth}
        \bmat{\mbf U & \mbf z(\sigma)}\bmat{{\xi} \\ M_0} = \mbf b(\sigma).
    \end{equation}
    Further, let $M_1 \in \bb C$ and assume that the data is informative for interpolation at $\sigma$.  Then the data is informative for Hermite interpolation (i.e., we may recover $M_1 = H'(\sigma)$ from the data) if and only if there exists $\xi_1 \in \bb C^p$ such that
    \begin{equation}
        \label{eq:CalcM1Orth}
        \bmat{\mbf U & \mbf z(\sigma)}\bmat{\xi_1 \\ M_1} = \mbf b^{(1)}(\sigma).
    \end{equation}
    Moreover, the data $(\bb U,\bb Y)$ are informative for interpolation at $\sigma$ if and only if
    \begin{subequations}
        \begin{equation}
            \label{eq:ExistCondOrth}
            \rank\left(\bmat{\mbf U & \mbf z(\sigma) & \mbf b(\sigma)}\right) = \rank\left(\bmat{\mbf U & \mbf z(\sigma)}\right),~\mbox{and}
        \end{equation}
        \begin{equation}
            \label{eq:UniqueCondOrth}
            \rank\left(\bmat{\mbf U & \mbf z(\sigma)}\right) = \rank\left({\mbf U}\right) + 1,
        \end{equation}
    \end{subequations}
    and the data are informative for Hermite interpolation at $\sigma$ if and only if the data are informative for interpolation at $\sigma$ and
    \begin{equation}
        \label{eq:ExistCondDerOrth}
        \rank\left(\bmat{\mbf U & \mbf z(\sigma) & \mbf b^{(1)}(\sigma)}\right) = \rank\left(\bmat{\mbf U & \mbf z(\sigma)}\right).
    \end{equation}
\end{lemma}
For a proof of \cref{prop:ExistUniqueOrth}, see \cite{AckermannG2023FreqROMTD}. To recover $M_0 \approx H(\sigma)$ from $\bb U,\bb Y$ using the methodology of \cite{burohmanBSC2020informativity} and \cite{AckermannG2023FreqROMTD}, one first checks the rank conditions \cref{eq:ExistCondOrth} and \cref{eq:UniqueCondOrth}, then calculates $M_0$ from \cref{eq:CalcM0Orth}. The derivative $M_1$ is calculated in a similar way; see \cite{AckermannG2023FreqROMTD} for details.
\begin{remark}
    \label{rem:dontNeedTrueN}
    \cref{prop:ExistUniqueOrth} assumes that the underlying system order, $n$, is known.  In \cite{AckermannG2023FreqROMTD}, we demonstrated that one may replace `$n$' with any integer $\hat n \geq n$ in \cref{prop:ExistUniqueOrth} and still maintain the same theoretical solution for $M_0$. We refer the reader to \cite{AckermannG2023FreqROMTD} for further details and proof.  The implications of using $\hat n > n$ will be explored later in this work in \cref{sec:CondNumFromSysProps} and \cref{sec:systemIndependentAnalysis}.
\end{remark}

\cref{prop:ExistUniqueOrth} provides linear systems that need to be solved for frequency information $M_0 \approx H(\sigma)$ and $M_1 \approx H'(\sigma)$.  The numerical results of \cite{AckermannG2023FreqROMTD} show that when $\sigma$ is restricted to the unit circle, this linear system can typically be solved accurately in double precision floating point arithmetic.  However, as discussed in \cref{sec:H2OptConds}, if one wishes to construct $\mc H_2$ optimal DDROMs from the data $(\bb U,\bb Y)$ we will require sampling where $|\sigma| > 1$. The next section first demonstrates the need for additional consideration of the linear system \cref{eq:CalcM0Orth}, then provides the necessary analysis to remedy the issues that arise.

\section{Optimizing the condition number}
\label{sec:Conditioning}
Note that recovering the frequency domain information
$M_0$ and $M_1$ from data requires solving the linear systems~\cref{eq:CalcM0Orth} and~\cref{eq:CalcM1Orth} with the same coefficient matrix $\bmat{\mbf U & \mbf z(\sigma)}$. 
In this section, we first examine the condition number of this matrix (\cref{sec:motivation}).  We will illustrate that this condition number grows rapidly as $|\sigma|$ increases greater than $1$, which will precisely be the case in the $\mc H_2$ optimal modeling problem.  We then present an explicit formula for the condition number of matrices that are composed of a subunitary block and a single additional column by providing an explicit formula for the eigenvalues of a rank 1 update to an orthogonal projection (\cref{sec:explicitCondNumForm}).  This leads to an optimal scaling to minimize the condition number of $\bmat{\mbf U & \delta \, \mbf z(\sigma)}$ for $\delta \in \bb R$.  Finally, in \cref{sec:CondNumFromSysProps}, we provide a connection between properties of the underlying system that generated the data $(\bb U,\bb Y)$ and the conditioning of $\bmat{\mbf U & \mbf z(\sigma)}$, which again provides guidance on how one may reduce the condition number of $\bmat{\mbf U & \mbf z(\sigma)}$ further.  We stress that while our analysis leads to modifications of \cref{eq:CalcM0Orth} to reduce its condition number, all of our modifications maintain the same theoretical solutions to the linear systems~\cref{eq:CalcM0Orth} 
and~\cref{eq:CalcM1Orth}.

\subsection{Motivation}
\label{sec:motivation}
Recall from \cref{alg:IRKA} and \cref{thm:H2optCond} that sampling $H$ at $1/\lambda$, where $\lambda$ is an eigenvalue of the matrix pencil $z\mbf E_r-\mbf A_r$ is required to construct $\mc H_2$ optimal DDROMs of $H$.  If we assume that the reduced model $\widehat H$ is asymptotically stable, then $|1/\lambda| > 1$.
Recall from \cref{eq:CalcM0Orth} that to recover $M_0 \approx H(\sigma)$ from data $(\bb U,\bb Y)$, we must solve a linear system involving the matrix $\bmat{\mbf U & \mbf z(\sigma)}$, where $\mbf U$ is subunitary and $\mbf z(\sigma)$ is given in \cref{eq:zSigmaDef}.  Since $\mbf z(\sigma)$ contains the values $\sigma^k, k=0,1,\ldots,n$, when $|\sigma| > 1$ and $n$ is moderately large, we will have $\|\mbf z(\sigma)\| > \sigma^n \gg 1$.  So compared to the normalized columns of $\mbf U$, the vector $\mbf z(\sigma)$ is differently scaled.  Intuitively, this should lead to ill-conditioning of the matrix $\bmat{\mbf U & \mbf z(\sigma)}$.

\begin{figure}[!htb]
    \centering
    \includegraphics[width = .8\textwidth]{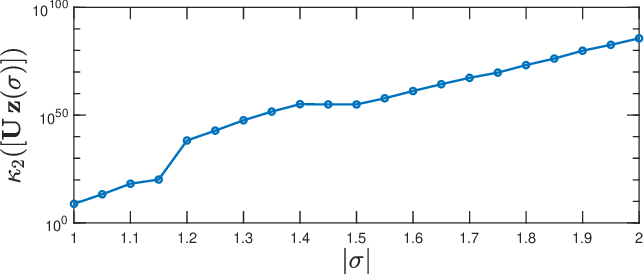}
    \caption{Condition number of $\bmat{\mbf U & \mbf z(\sigma)}$ grows rapidly with $|\sigma|$}
    \label{fig:ConditionVsRadius}
\end{figure}

\cref{fig:ConditionVsRadius} explores this issue numerically.  We consider the ISS1r system from \cite{Chahlaoui2002Benchmarks}.  This system has order $n = 270$, is time-continuous, and has 3 inputs and 3 outputs.  To adapt it to our SISO, discrete-time setting we consider only the first input and first output, and discretize the model with sampling frequency $10^2$ Hz with a zero-order hold method (\MATLAB's \texttt{c2d} command).  We call the resulting discrete-time SISO system $\mc S_1$ and its  transfer function $H_1(z)$.   We simulate $\mc S_1$ for $T = 3n$ time steps with the Gaussian random input $\bb U \in \bb R^{T+1}$ to obtain the output $\bb Y \in \bb R^{T+1}$. From this data we form $\mbf G_n$ as \cref{eq:GnDef} and set $\mbf U = \orth(\mbf G_n)$.  Then, for $\omega = 0.5$ and for each $d = 1,1.05,\ldots,2$, we set $\sigma = de^{\omega \mbf i}$ and form $\mbf z(\sigma)$ as in~\cref{eq:zSigmaDef}.  \cref{fig:ConditionVsRadius} shows how the condition number of $\bmat{\mbf U & \mbf z(\sigma)}$ varies with $|\sigma| = d$.  As expected, as $|\sigma|$ increases, the condition number of $\bmat{\mbf U & \mbf z(\sigma)}$ increases rapidly. Note that this value of $n = 270$ is quite small for many reduced order modeling problems, and for larger $n$, $\kappa_2(\bmat{\mbf U & \mbf z(\sigma)})$ increases even more rapidly than shown in \cref{fig:ConditionVsRadius}.

Clearly, if we are to sample $H$ and $H'$ at $|\sigma| > 1$, solving \cref{eq:CalcM0Orth} and \cref{eq:CalcM1Orth} as they are written are most likely produce inaccurate results especially when the data is noisy.  Since the cause of the ill-conditioning appears to be the scaling of $\mbf z(\sigma)$, a possible solution is to normalize $\mbf z(\sigma)$ as a preconditioning step.  \Cref{sec:explicitCondNumForm} expands upon this intuition, showing that normalizing $\mbf z(\sigma)$ does in fact provides the optimal scaling $\delta\mbf z(\sigma)$ that minimizes $\kappa_2(\bmat{\mbf U & \delta \mbf z(\sigma)})$.  This is accomplished by deriving an explicit formula for the condition number of any subunitary matrix appended by an extra column.

\subsection{An explicit condition number formula}

\label{sec:explicitCondNumForm}
We now present an explicit formula (\cref{prop:CondNumFormula}) for the condition number of any matrix with form
$\begin{bmatrix}\mbf Q & \mbf z\end{bmatrix},$
where $\mbf Q \in \bb C^{m \times n}$ is subunitary (so $n < m$) and $\mbf z \in \bb C^m$ is arbitrary.  We will also prove that choosing $\delta = 1/\|\mbf z\|$ minimizes the condition number
$\kappa_2\left(\begin{bmatrix}\mbf Q & \delta \mbf z\end{bmatrix}\right)$
over all numbers $\delta \in \bb R$.  Our strategy will be to analyze the eigenvalues of $\bmat{\mbf Q & \mbf z}\bmat{\mbf Q & \mbf z}^* = \mbf Q\mbf Q^*+\mbf z\mbf z^*$,
 and investigate how they change when $\|\mbf z\|$ is varied.

To begin, note that $\mbf Q \mbf Q^* + \mbf z \mbf z^* \in \bb C^{m\times m}$ has at most $n+1$ nonzero eigenvalues.  Indeed, if $\mbf z$ does not lie in the range of $\mbf U$, then $\mbf Q \mbf Q^* + \mbf z \mbf z^*$ has exactly $n+1$ nonzero eigenvalues.  Recognizing that $\mbf Q \mbf Q^* + \mbf z \mbf z^*$ is a rank 1 Hermitian perturbation of a Hermitian matrix, Weyl's theorem \cite{SymetricEigenProb_Bounds, weylAsymptotischeVerteilungsgesetzEigenwerte1912} provides bounds on the eigenvalues of $\mbf Q \mbf Q^* + \mbf z \mbf z^*$.

\begin{theorem}[Weyl's theorem \cite{SymetricEigenProb_Bounds, weylAsymptotischeVerteilungsgesetzEigenwerte1912}]
    \label{thm:WeylsThm}
    Let $\mbf A,\mbf B \in \bb C^{m \times m}$ be Hermitian and let $i,j = 1,2,\ldots,m$.  Then 
    \begin{subequations}
        \begin{equation}
            \label{eq:Weyl1}
            \lambda_{i + j-1}(\mbf A + \mbf B) \leq \lambda_i(\mbf A) + \lambda_j(\mbf B); \quad i + j \leq m+1, \quad \text{and}
        \end{equation}
        \begin{equation}
            \label{eq:Weyl2}
            \lambda_{i + j-m}(\mbf A + \mbf B)
            \geq 
            \lambda_i(\mbf A) + \lambda_j(\mbf B); \quad i + j  \geq m+1.
        \end{equation}
    \end{subequations}
    
\end{theorem}

Specializing $i = 1,2,\ldots,n$ and $j = 2$ in \cref{eq:Weyl1} to our setting shows that 
\begin{equation}
    \label{eq:WeylLowerBound}
    \lambda_k(\mbf Q\mbf Q^* + \mbf z \mbf z^*) \leq \lambda_i(\mbf Q\mbf Q^*) + \lambda_j(\mbf z \mbf z^*) = 1; \quad k = 2,3,\ldots,n+1.
\end{equation}
Specializing $i = 1,2,\ldots, n$ and $j = m$ in \cref{eq:Weyl2} shows that
\begin{equation}
    \label{eq:WeylUpperBound}
    \lambda_k(\mbf Q\mbf Q^* + \mbf z \mbf z^*) \geq \lambda_i(\mbf Q\mbf Q^*) + \lambda_j(\mbf z \mbf z^*) = 1; \quad k = 1,2,\ldots,n.
\end{equation}
Assuming that $\mbf z$ does not lie in the range of $\mbf Q$, \cref{eq:WeylLowerBound} and \cref{eq:WeylUpperBound} show that $\mbf Q\mbf Q^* + \mbf z \mbf z^*$ has exactly $n-1$ eigenvectors with eigenvalue 1, one eigenvector with eigenvalue at least 1, and one eigenvector with eigenvalue at most 1.  For the purpose of calculating the condition number of $\bmat{\mbf Q & \mbf z}$, we must find these extreme eigenvalues.  The following lemma shows where these two eigenvectors corresponding to the smallest and largest eigenvalue of $\mbf Q\mbf Q^* + \mbf z \mbf z^*$ must lie.

\begin{lemma}
    \label{lem:eigenvectorsInSpanUV}
    Let $\mbf Q \in \bb C^{m \times n}$ with $m>n$ be subunitary.  Let $\mbf z \in \bb C^{m}$.  Define
        $$\mbf u = \mbf Q\mbf Q^*\mbf z \in \bb C^{m}  \quad \text{and} \quad \mbf v = (\mbf I-\mbf Q\mbf Q^*)\mbf z\in \bb C^{m}$$
        so that $\mbf u + \mbf v = \mbf z$.  Assume $\mbf v \neq 0$ and $\mbf u \neq 0$ (i.e., $\mbf z$ is not in $\Range(\mbf Q)$ or orthogonal to $\Range(\mbf Q)$).
        Then the matrix $\mbf Q \mbf Q^*+\mbf z\mbf z^*$ has $n-1$ eigenvectors with eigenvalue $1$ that lie in $\Range(\mbf Q)\setminus\myspan(\mbf u)$ and  two (non-zero) extreme eigenvectors, ${\mbf x}_1$ and ${\mbf x}_{n+1}$, with the corresponding eigenvalues $\lambda_1 \geq 1$ and $0<\lambda_{n+1} \leq 1$, that lie in $\myspan\{\mbf u,\mbf v\}$.  The remaining $n-1$ eigenvalues are $0$.
\end{lemma}
\begin{proof}
    Since $\mbf Q \mbf Q^*+\mbf z\mbf z^*$ is Hermitian, it is unitarily diagnalizable as 
    $\mbf X \Lambda \mbf X^*$ where
    $\mbf X \in \bb C^{m \times m}$ is unitary
    and $\Lambda \in \bb R^{m\times m}$ is diagonal.
    Since $\mbf v \neq 0$ and $\mbf Q$ is subunitary (hence full column rank), there are $n+1$ nonzero eigenvalues of $\mbf Q \mbf Q^*+\mbf z\mbf z^*$.
    Also since $\mbf Q$ has full column rank, we can find $n-1$ vectors $\mbf u_i \in \Range(\mbf Q)$ such that 
    $\mbf u_i^*\mbf u = 0,\quad i = 1,2,\ldots,n-1.$
    Since each $\mbf u_i \in \Range(\mbf Q)$, we also have
    $\mbf u_i^*\mbf v = 0$, for $i = 1,2,\ldots,n-1.$
    Note that each $\mbf u_i$ is an eigenvector of $\mbf Q \mbf Q^*+\mbf z\mbf z^*$ with the corresponding eigenvalue $\lambda_i = 1$:
    $$(\mbf Q \mbf Q^*+\mbf z\mbf z^*)\mbf u_i = (\mbf Q \mbf Q^*+(\mbf u + \mbf v)(\mbf u + \mbf v)^*)\mbf u_i = \mbf u_i.$$
    Thus we have found the $n-1$ eigenvectors with eigenvalue $\lambda = 1$ in $\Range(\mbf Q)\setminus\myspan(\mbf u)$, so the remaining 2 non-zero eigenvectors are $\mbf x_1$ and $\mbf x_{n+1}$ with eigenvalues $\lambda_1 \geq 1$ and $0 < \lambda_{n+1} \leq 1$ (by \cref{thm:WeylsThm}).  Now note that since $\mbf u \neq 0$ and $\mbf v \neq 0$, neither $\mbf u$ nor $\mbf v$ is an eigenvector of $\mbf Q \mbf Q^*+\mbf z\mbf z^*$:
    \begin{equation}
        (\mbf Q \mbf Q^*+\mbf z\mbf z^*)\mbf u = (\mbf Q \mbf Q^*+(\mbf u + \mbf v)(\mbf u + \mbf v)^*)\mbf u = \mbf u+\|\mbf u\|^2(\mbf u + \mbf v),
    \end{equation}
    \begin{equation}
        \label{eq:vNotEigVector}
        (\mbf Q \mbf Q^*+\mbf z\mbf z^*)\mbf v = (\mbf Q \mbf Q^*+(\mbf u + \mbf v)(\mbf u + \mbf v)^*)\mbf v = \|\mbf v\|^2(\mbf u + \mbf v).
    \end{equation}
    But both $\mbf u$ and $\mbf v$ are in $\Range(\mbf Q\mbf Q^* + \mbf z \mbf z^*)$. The manipulations
    \begin{equation}
        \label{eq:uInRange}
        (\mbf Q \mbf Q^*+\mbf z\mbf z^*)\left(\frac{\mbf u}{\|\mbf u\|^2} - \frac{\mbf v}{\|\mbf v\|^2}\right) = (\mbf Q \mbf Q^*+(\mbf u + \mbf v)(\mbf u + \mbf v)^*)\left(\frac{\mbf u}{\|\mbf u\|^2} - \frac{\mbf v}{\|\mbf v\|^2}\right) = \frac{\mbf u}{\|\mbf u\|^2},
    \end{equation}
    show $\mbf u \in \Range(\mbf Q\mbf Q^* + \mbf z \mbf z^*)$,
    and \cref{eq:vNotEigVector} and \cref{eq:uInRange} together show $\mbf v \in \Range(\mbf Q\mbf Q^* + \mbf z \mbf z^*)$.
    Since the $n-1$ eigenvectors of $\mbf Q \mbf Q^*+\mbf z\mbf z^*$ with eigenvalue 1 are orthogonal to $\mbf u$ and $\mbf v$, and both $\mbf u \in \Range(\mbf Q \mbf Q^*+\mbf z\mbf z^*)$ and $\mbf v \in \Range(\mbf Q \mbf Q^*+\mbf z\mbf z^*)$ while not being eigenvectors; they both must be in the span of $\mbf x_1$ and $\mbf x_{n+1}$.  Since $\mbf x_1$ and $\mbf x_{n+1}$ are eigenvectors of a diagonalizable matrix they are linearly independent, so
    $\myspan\{\mbf u,\mbf v\} = \myspan\{\mbf x_1, \mbf x_{n+1}\},$
    which implies
    $\mbf x_1 \in \myspan\{\mbf u,\mbf v\}$ and  $\mbf x_{n+1} \in \myspan\{\mbf u,\mbf v\}.$
\end{proof}
We now use \cref{lem:eigenvectorsInSpanUV} to prove an explicit formula for the condition number of $\bmat{\mbf Q & \mbf z}$. 
\begin{theorem} 
    \label{prop:CondNumFormula}
    Let $\mbf Q \in \mathbb C^{m \times n}$ with $ n< m$ be subunitary. Let $\mathbf z \in \mathbb C^m$, $\mathbf u = \mbf Q\mbf Q^*\mathbf z$, $\mathbf v = (\mathbf I-\mbf Q\mbf Q^*)\mathbf z$, and $\nu = \|\mathbf z\|$.  Assume $\|\mbf v\| \neq 0$ and $\|\mbf u\| \neq 0$.  Then 
    $\lambda_1 = \frac{1}{2}\left(1+ \nu^2 + \sqrt{1+\nu^4+2\nu^2-4\|\mathbf v\|^2}\right) \geq 1$ is the largest eigenvalue of 
    $\mbf Q\mbf Q^* + \mathbf z\mathbf z^*$ and $
       0 < \lambda_{n+1} = \frac{1}{2}\left(1+ \nu^2 - \sqrt{1+\nu^4+2\nu^2-4\|\mathbf v\|^2}\right) \leq 1$ is 
        the smallest nonzero eigenvalue. Thus,
    the matrix $\bmat{\mbf Q & \mbf z}$ the has condition number
    \begin{equation}
        \label{eq:K2EqExplicit}
        \kappa_2\left(\bmat{\mbf Q & \mbf z}\right) = \sqrt{\frac{{1+ \nu^2 + \sqrt{1+\nu^4+2\nu^2-4\|\mathbf v\|^2}}}{{1+ \nu^2 - \sqrt{1+\nu^4+2\nu^2-4\|\mathbf v\|^2}}}}.
    \end{equation}
    
\end{theorem}

\begin{proof}
    We will consider the eigenvalues of $\bmat{\mbf Q & \mbf z}\bmat{\mbf Q & \mbf z}^* = \mbf Q\mbf Q^* + \mathbf z\mathbf z^*$.  First we expand
    \begin{equation}
        \begin{aligned}
            \mbf Q\mbf Q^*+\mathbf z\mathbf z^* &= \mbf Q\mbf Q^*+(\mathbf u+\mathbf v)(\mathbf u+\mathbf v)^*\\
            &=\mbf Q\mbf Q^*+\mathbf u\mathbf u^*+\mathbf v\mathbf v^*+\mathbf u\mathbf v^*+\mathbf v\mathbf u^*.
        \end{aligned}
    \end{equation}
    By \cref{lem:eigenvectorsInSpanUV}, 
    the eigenvectors of $\mbf Q\mbf Q^* + \mathbf z\mathbf z^*$ corresponding to the extreme eigenvalues $\lambda_1$ and $\lambda_{n_+1}$  are in $\myspan\{\mbf u,\mbf v\}$. So for some $\alpha_i \in \bb C$ and $ \beta_i \in \bb C, $ and $ \lambda_i \in \mathbb R$, for $i = 1$ and $i=n+1$, we have
    \begin{equation}
        \label{eq:aubvIsEigenvector}
        (\mbf Q\mbf Q^*+\mathbf u\mathbf u^*+\mathbf v\mathbf v^*+\mathbf u\mathbf v^*+\mathbf v\mathbf u^*)(\alpha_i \mathbf u+ \beta_i \mathbf v) = \lambda_i(\alpha_i \mathbf u+ \beta_i \mathbf v),
    \end{equation}
    i.e, $(\alpha_i \mathbf u+ \beta_i \mathbf v, \lambda_i)$ are eigenpairs for $i = 1$ and $i=n+1$. Note that since $\mbf u$ and $\mbf v$ are not eigenvectors, neither $\alpha_i$ nor $\beta_i$ are 0.  Now expanding \cref{eq:aubvIsEigenvector} (and dropping subscripts), we obtain
    \begin{equation}
        \label{eq:expandedAUBV}
        \alpha \mathbf u + \alpha \|\mathbf u\|^2\mathbf u + \alpha\|\mathbf u\|^2\mathbf v + \beta\|\mathbf v\|^2\mathbf v + \beta \|\mathbf v\|^2\mathbf u = \lambda(\alpha \mathbf u + \beta \mathbf v),
    \end{equation}
    and grouping the terms of \cref{eq:expandedAUBV} yields
    \begin{equation}  \label{eq:uandv}
        ((1-\lambda)\alpha + \alpha \|\mathbf u\|^2+\beta \|\mathbf v\|^2)\mathbf u + (\alpha \|\mathbf u\|^2 + \beta \|\mathbf v\|^2 - \lambda \beta)\mathbf v = 0.
    \end{equation}
    Since $\mathbf u$ and $\mathbf v$ are orthogonal to each other (and thus are linearly independent),~\eqref{eq:uandv} yields
    \begin{equation}
        \begin{aligned}
            \label{eq:AUBVSys}
            (1-\lambda)\alpha + \alpha \|\mathbf u\|^2+\beta \|\mathbf v\|^2  &= 0\qquad \mbox{and} \qquad \alpha \|\mathbf u\|^2 + \beta \|\mathbf v\|^2 - \lambda \beta &= 0.
        \end{aligned}
    \end{equation}
    Since an eigenvector can be scaled, we are free to  choose $\alpha$ or $\beta$.  Setting $\beta = 1$ and solving \cref{eq:AUBVSys} for $\lambda$ and $\alpha$ yields
    \begin{subequations}
        \begin{equation}
            \alpha = \frac{1}{2\|\mathbf u\|^2}\left(1-\|\mathbf v\|^2+\|\mathbf u\|^2\pm \sqrt{(\|\mathbf v\|^2-\|\mathbf u\|^2-1)^2+4\|\mathbf u\|^2\|\mathbf v\|^2}\right),~\mbox{and}
        \end{equation}
        \begin{equation}
            \label{eq:LamUnsimp}
            \lambda = \frac{1}{2}\left(1-\|\mathbf v\|^2+\|\mathbf u\|^2\pm \sqrt{(1-\|\mathbf v\|^2+\|\mathbf u\|^2)^2+4\|\mathbf u\|^2\|\mathbf v\|^2}\right) + \|\mathbf v\|^2.
        \end{equation}
    \end{subequations}
    Recalling that $\nu^2 = \|\mbf z\|^2 = \|\mbf u\|^2 + \|\mbf v\|^2$, simplifying \cref{eq:LamUnsimp} leads to
    \begin{equation}
        \label{eq:eigUUhxxh}
        \lambda = \frac{1}{2}\left(1+ \nu^2 \pm \sqrt{1+\nu^4+2\nu^2-4\|\mathbf v\|^2}\right).
    \end{equation}

    \cref{lem:eigenvectorsInSpanUV} guarantees that the eigenvalues in \cref{eq:eigUUhxxh} are the the extreme eigenvalues, $\lambda_1$ and $\lambda_{n+1}$, which proves the formulae for 
    $\lambda_1$ and $\lambda_{n+1}$. Thus, we obtain 
    \begin{equation}
        \label{eq:condNumFormula}
        \sqrt{\frac{\lambda_1}{\lambda_{n+1}}} = \kappa_2\left([~\mbf Q ~\mbf z]\right) = \sqrt{\frac{{1+ \nu^2 + \sqrt{1+\nu^4+2\nu^2-4\|\mathbf v\|^2}}}{{1+ \nu^2 - \sqrt{1+\nu^4+2\nu^2-4\|\mathbf v\|^2}}}},
    \end{equation}
    which proves the condition number formula.
\end{proof}
\begin{remark}
    The expression for the value of the smallest eigenvalue of $\mbf Q \mbf Q^* + \mbf z \mbf z^*$ appears similar to the lower bound for the smallest eigenvalue of a perturbed Hermitian matrix in \cite{ipsenRefinedPerturbationBounds2009}.  We emphasize that we have presented a formula for the \emph{exact values}  of the {\emph{both extreme non-zero} eigenvalues of matrices of the form $\mbf Q \mbf Q^* + \mbf z \mbf z^*$,} with subnitary $\mbf Q \in \bb C^{m \times n}$ with  $m>n$, while \cite{ipsenRefinedPerturbationBounds2009} presents \emph{upper and lower bounds} for eigenvalues of $\mbf A + \mbf z \mbf z^*$, where $\mbf A = \mbf A^* \in \bb C^{m \times m}$. Only for the very specific case of $\mbf Q \in \bb C^{m \times (m-1)}$, (corresponding to only the case of $n=m-1$ in our analysis), the lower bound of \cite{ipsenRefinedPerturbationBounds2009}
    agrees with the value presented in \cref{eq:eigUUhxxh}, yet  the upper bound is still pessimistic. 
 \end{remark}
Now consider a {fixed nonzero $\mbf z$}.  Then, independent of $\|\mbf z\|$, there exists $0 < \eta< 1$ such that $\|\mbf v\| = \eta\|\mbf z\| = \eta \nu$.  Substituting $\|\mbf v\| = \eta\nu$ in \cref{eq:condNumFormula} and squaring gives
\begin{equation}
    \label{eq:condNumSquared}
     \kappa^2_2\left([~\mbf Q ~\mbf z]\right) =
     K(\nu,\eta) = 
     \frac{1+ \nu^2 + \sqrt{1+\nu^4+2\nu^2-4\eta^2\nu^2}}{1+ \nu^2 - \sqrt{1+\nu^4+2\nu^2-4\eta^2\nu^2}}.
\end{equation}
Then differentiating with respect to $\nu$ gives
\begin{equation}
    \label{eq:condNumDerivative}
    \frac{\partial K}{\partial \nu}(\nu,\eta) = \frac{8\eta^2\nu(\nu^2-1)}{\sqrt{1+(2-4\eta^2)\nu^2 + \nu^4}\left(1+\nu^2-\sqrt{1+(2-4\eta^2)\nu^2 + \nu^4}\right)^2},
\end{equation}
which leads to \cref{prop:1isOptimal}.
\begin{corollary}
    \label{prop:1isOptimal}
    Let $\mbf Q \in \bb C^{m \times n}$ with $n < m$ be subunitary. Let $\mbf z \in \bb C^m$ be such that $\mbf z \not\in \Range(\mbf Q)$.  Then $\delta = \frac{1}{\|\mbf z\|}$ minimizes $\kappa_2(\bmat{\mbf Q & \delta \mbf z})$ over all $\delta \in \bb R$.
\end{corollary}
\begin{proof}
    If $\mbf z \perp \Range(\mbf Q)$, then normalizing $\mbf z$ leads to $\bmat{\mbf Q & \mbf z}$ being subunitary (or potentially unitary).  So $\kappa_2(\bmat{\mbf Q & \mbf z}) = 1$, which is clearly minimized.
    
    If $\mbf z \not\perp \Range(\mbf Q)$, then $\kappa_2(\bmat{\mbf Q & \mbf z})$ is given by \cref{eq:condNumFormula} and, when squared, has derivative \cref{eq:condNumDerivative}.
    Note that the denominator of \cref{eq:condNumDerivative} is always greater than zero, and $8\eta^2\nu$ is also always greater than zero by our assumptions.  So the sign of \cref{eq:condNumDerivative} is determined by 
    $\nu^2-1.$
    Since $\nu^2-1$ is zero if $\nu = 1$, negative if $0<\nu<1$, and positive for $\nu > 1$, the choice $\nu = 1$ is the only local minimum of \cref{eq:condNumFormula} for the feasible values of $\nu$ and $\eta$.
    Then since \cref{eq:condNumFormula} is decreasing for all allowed values of $\nu$ to the left of 1 and increasing for all $\nu > 1$, $\nu = 1$ is the global minimum.
\end{proof}
 In \cref{fig:ModulusVsCond_normalized} we illustrate the impact of the optimal scaling in~\cref{prop:1isOptimal}  by repeating the experiment in \cref{sec:motivation}, but recording the condition number  of $\bmat{\mbf U & \hat{\mbf z}(\sigma)}$ where $\hat{\mbf z}(\sigma) = \mbf z(\sigma)/\|\mbf z(\sigma)\|$ as $|\sigma|$ increases. The first observation is that $\kappa_2(\bmat{\mbf U & \hat{\mbf z}})$ is significantly reduced compared to \cref{fig:ConditionVsRadius}. Moreover, remarkably, the condition number {with optimal scaling} actually \emph{decreases} as $|\sigma|$ increases, as opposed to the condition number without scaling in \cref{fig:ConditionVsRadius}, which increased rapidly as $|\sigma|$ grew. The reason for this behavior is explained in \cref{sec:AvoidOverflow}, and taken advantage of in our algorithmic development later on.
\begin{figure}[!htb]
    \centering
    \includegraphics[width = .8\textwidth]{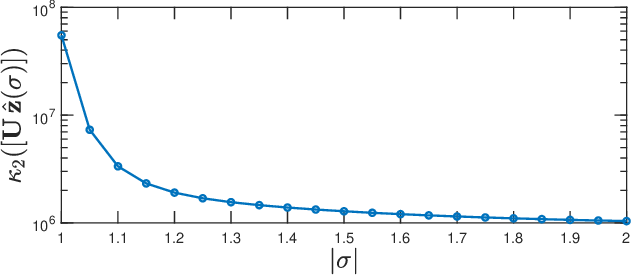}
    \caption{Condition number of $\bmat{\mbf U & \hat{\mbf z}(\sigma)}$ as $|\sigma|$ increases.}
    \label{fig:ModulusVsCond_normalized}
\end{figure}

This analysis leads us to replace $\mbf z(\sigma)$ with $\hat{\mbf z}(\sigma) \coloneqq \mbf z(\sigma)/\|\mbf z(\sigma)\|$ in \cref{eq:CalcM0Orth} and 
\begin{equation}
    \begin{bmatrix}
        \label{eq:calcM0Normalize}
        \mbf U & \hat{\mbf z}(\sigma)
        \end{bmatrix}
        \bmat{{\xi} \\ \hat M_0} = \mbf b(\sigma); \quad M_{0} = \frac{1}{\|\mbf z\|}\hat M_{0},
\end{equation}
which allows to recover $M_0$ by solving a much better conditioned problem. We perform an analogous process for $M_1$.
\subsection{Predicting conditioning from system properties}
\label{sec:CondNumFromSysProps}
\Cref{sec:explicitCondNumForm} provided an explicit formula for the condition number of the matrix in \cref{eq:CalcM0Orth}, which motivated preconditioning by normalizing $\mbf z(\sigma)$, shown in \cref{eq:calcM0Normalize}.  Now, with $\|\hat{\mbf z}(\sigma)\|$ fixed, the condition number of \cref{eq:calcM0Normalize} is dependent only on the angle between $\Range(\mbf U)$ and $\hat{\mbf z}(\sigma)$ (the normalized $\mbf z(\sigma)$).  This section shows how we may analyze \cref{eq:CalcM1Orth} to reveal how properties of the underlying system $\mc S$ as in \cref{eq:LinDiscSys} effect the conditioning of \cref{eq:calcM0Normalize}.

Let $\mathcal S$ be a system as in \cref{eq:LinDiscSys}, $\bb U \in \bb R^{T+1}$ be an input to $\mathcal S$, and $\bb Y \in \bb R^{T+1}$ be the corresponding output.  Assume that for $\bb U, \bb Y$, and $\sigma \in \bb C$ the conditions \cref{eq:ExistCondOrth}, \cref{eq:UniqueCondOrth}, and \cref{eq:ExistCondDerOrth} are all satisfied.  Then, {recalling \cref{eq:CalcM1Orth}}, there exists $\xi_1 \in \bb C^{T-n}$ such that

$$\begin{bmatrix}
\mbf U & \mbf z(\sigma)
\end{bmatrix}
\begin{bmatrix}
\xi_1 \\ H'(\sigma)
\end{bmatrix} = 
\begin{bmatrix}
\gamma_n^{(1)}(\sigma)\\ 
H(\sigma)\gamma^{(1)}_n(\sigma)
\end{bmatrix},
$$
where $\gamma^{(1)}_n(\sigma)$ is as defined in~\cref{eq:gamma1_def}.  Decompose $\mbf z(\sigma)$ as
\begin{equation}
    \mbf z(\sigma) = 
    \mbf u + \mbf v, \quad \mbf u \in \Range(\mbf U), \quad \mbf U^{\top} \mbf v = \mbf 0.
\end{equation}
Then there exists a (unique) $\hat \xi_1 \in \bb C^p$ such that
\begin{equation}
    \begin{bmatrix}
    \label{eq:SubWforZ}
    \mbf U & \mbf v
    \end{bmatrix}
    \begin{bmatrix}
    \hat \xi_1 \\ H'(\sigma)
    \end{bmatrix} = 
    \begin{bmatrix}
    \gamma_n^{(1)}(\sigma)\\ 
    H(\sigma)\gamma^{(1)}_n(\sigma)
    \end{bmatrix}.
\end{equation}
Note that the columns of $\mbf U$ are orthogonal to $\mbf v$ and thus
\begin{equation}
    \label{eq:NormEqUnsimp}
    \|\mbf U\hat \xi_1\|^2 + \|H'(\sigma) \mbf v \|^2 = 
    \left\|\begin{bmatrix} 
    \gamma_n^{(1)}(\sigma)\\ 
    H(\sigma)\gamma^{(1)}_n(\sigma)
    \end{bmatrix}\right\|^2,
\end{equation}
which may be simplified to
\begin{equation}
    \label{eq:NormEqUnsimp2}
    \|\hat \xi_1\|^2 + |H'(\sigma)|^2\|\mbf v \|^2 
    =(1+|H(\sigma)|^2)\|\gamma^{(1)}_n(\sigma)\|^2.
\end{equation}
\begin{comment}
    Given any vector $\mbf x \in \bb C^m$ and any scalar $\beta \in \bb C$, we have
    \begin{equation}
        \left\|\begin{bmatrix}
        \mbf x\\ 
        \beta\mbf x
        \end{bmatrix}\right\|^2
        =
        \sum_{i = 1}^m (|x_i|^2 + (|\beta x_i|)^2) = \sum_{i=1}^m (1+|\beta|^2)|x_i|^2 = (1+|\beta|^2)\|\mbf x\|^2.
    \end{equation}
    So \cref{eq:NormEqUnsimp} becomes
    \begin{equation}
        \label{eq:NormEqSimp}
        \|\tilde \xi_1 \mbf U_1 \|^2 + \ldots + \|\tilde \xi_k \mbf U_k \|^2 + \|H'(\sigma) \mbf w \|^2 = (1+|H(\sigma)|^2)\|\gamma^{(1)}_n(\sigma)\|^2
    \end{equation}
\end{comment}
Now let $\theta \in (0,\pi/2]$ be the angle that $\mbf z(\sigma)$ makes with the subspace spanned by $\mbf U$. Note that $0$ is not included in the interval as $\theta = 0$ would imply that $\mbf z(\sigma) \in \Range(\mbf U)$, which violates \cref{eq:UniqueCondOrth}.
Define $\alpha = \sin(\theta)$, so $0 < \alpha \leq 1$ and 
$\alpha = \frac{\|\mbf v\|}{\|\mbf z(\sigma)\|} = \frac{\|\mbf v\|}{\|\gamma_n(\sigma)\|}.$
Then we can write \cref{eq:NormEqUnsimp2} as
\begin{equation}
    \label{eq:NormEqWithAlpha}
    \|\hat \xi_1\|^2 + |H'(\sigma)|^2\alpha^2\|\gamma_n(\sigma)\|^2 = (1+|H(\sigma)|^2)\|\gamma_n^{(1)}(\sigma)\|^2.
\end{equation}
Since $\alpha > 0$, solving \cref{eq:NormEqWithAlpha} for $\alpha$ gives
\begin{equation}
    \label{eq:alphaEquation}
    \alpha = \sqrt{\frac{(1+|H(\sigma)|^2)\|\gamma_n^{(1)}(\sigma)\|^2 - \|\hat \xi_1\|^2}{|H'(\sigma)|^2\|\gamma_n(\sigma)\|^2}}.
\end{equation}
\begin{remark}
    Since $\alpha$ is defined as $\frac{\|\mbf v\|}{\mbf z(\sigma)}$, and any scaling done to $\mbf z(\sigma)$ will also be done to $\mbf v$, we see that the value of $\alpha$ is unchanged by the normalization step in \cref{eq:calcM0Normalize}.  Thus the following analysis applies to both \cref{eq:CalcM0Orth} and \cref{eq:calcM0Normalize}.
\end{remark}
Since $\mbf U$ has orthonormal columns and $\|\hat{\mbf z}(\sigma)\| = 1$, $\alpha$ will completely determine
$\mathcal{K} = \kappa_2\left(\begin{bmatrix}
\mbf U & \hat{\mbf z}(\sigma)
\end{bmatrix}\right).$  In particular, as $\alpha = \sin(\theta)$ decreases (to zero), $\hat{\mbf z}(\sigma)$ comes close to  $\Range(\mbf U)$ and $\mathcal{K}$ increases.  And vice verso, as $\alpha$ increases to $1$, $\hat{\mbf z}(\sigma)$ becomes nearly orthogonal to  $\Range(\mbf U)$ and  $\mathcal{K}$ decreases.

In the coming sections, we aim to connect the properties of the underlying system $\mc S$ and parameters of the data-informativity framework to the value of $\alpha$ \cref{eq:alphaEquation}, and thus the conditioning of \cref{eq:calcM0Normalize}.  To this end, we identify two important groups of quantities present in \cref{eq:alphaEquation}.
The first group are quantities which are intrinsic to the system $\mc S$.  These quantities are $H(\sigma)$ and $H'(\sigma)$, which we call ``system-dependent quantities". Importantly, these quantities do not depend on parameters of the data-informativity framework.  The other group are the ``system-independent quantities", $\|\gamma_n(\sigma)\|$ and $\|\gamma_n^{(1)}(\sigma)\|$, which 
{are artifacts of the data-informativity framework.}
  We first investigate how the system-dependent quantities effect conditioning of \cref{eq:calcM0Normalize} in \cref{sec:sysDependentAnalysis}, then examine how the system-independent quantities effect conditioning of \cref{eq:calcM0Normalize} in \cref{sec:systemIndependentAnalysis}.  The influence of $\|\hat \xi_1\|$ is discussed in \cref{rem:XiNotControlled}.

\subsubsection{System dependent analysis}
\label{sec:sysDependentAnalysis}

Consider the recovery of $H(\sigma)$ for values of $\sigma$ that lie equidistant from the origin (e.g., each $\sigma$ lies on the unit circle).  In this scenerio, the quantities $\|\gamma_n(\sigma)\|$ and $\|\gamma^{(1)}_n(\sigma)\|$ do not change, but $H(\sigma)$ and $H'(\sigma)$ do change.
Assuming that $\|\hat \xi_1\|$ remains constant (see remark \ref{rem:XiNotControlled}), we observe that as the ratio $|H'(\sigma)|/|H(\sigma)|$ increases (decreases), $\alpha$ decreases (increases), which causes the condition number of $\bmat{\mbf U & \hat{\mbf z}(\sigma)}$ to increase (decrease).  This implies that, in general, if we attempt to recover $H(\sigma)$ where $|H'(\sigma)|/|H(\sigma)|$ is small, we should expect the matrix $\bmat{\mbf U & \hat{\mbf z}(\sigma)}$ to be better conditioned than if we attempt to recover $H(\sigma)$ where $|H'(\sigma)|/|H(\sigma)|$ is large. ({This discussion agrees with the classical definition of a condition number of a problem of evaluating $f(\sigma) = H(\sigma)$ as we assume $\mid \sigma \mid = 1$.})

To examine the relationship between the relative derivative $|H'(\sigma)|/|H(\sigma)|$ and conditioning of $\bmat{\mbf U & \hat{\mbf z}(\sigma)}$, we consider the ISS1R model, $\mc S_1$ with transfer function $H_1(z)$ described in \cref{sec:motivation}.    We simulate $\mc S_1$ for $T = 3n$ time steps with a Gaussian random input $\bb U \in \bb R^{T+1}$ to obtain the output $\bb Y \in \bb R^{T+1}$, and form $\mbf G_n$ as in \cref{eq:GnDef}, $\mbf U = \orth(\mbf G_n)$, and $\mbf z(\sigma)$ as in \cref{eq:zSigmaDef}.  We record $H_1'(\sigma_i)/H_1(\sigma_i)$ (the true values) for $ i=1,2,\ldots,1,\!000$ and $\sigma_i = e^{\omega_i \mbf i}$, where the $\omega_i$ are logarithmically spaced in $[10^{-3},\pi)$, as well as the condition number of $\bmat{\mbf U & \hat{\mbf z}(\sigma_i)}$ for each $i = 1,\ldots,1,\!000$.  These results are displayed in \cref{fig:CondNumVsRelDer}, which shows that, in general, as the relative derivative $H'_1(\sigma)/H_1(\sigma)$ decreases, so does the condition number of $\bmat{\mbf U & \hat{\mbf z}(\sigma_i)}$, thus supporting the analysis above.

\subsubsection{System independent analysis}
\label{sec:systemIndependentAnalysis}
Recall from \cite{AckermannG2023FreqROMTD} or \cref{rem:dontNeedTrueN} that any integer $\hat n > n$ used to calculate frequency information via \cref{eq:calcM0Normalize} will, theoretically, yield the same result as using the true order of the underlying system, $n$.  If such a $\hat n > n$ is used, then $H(\sigma)$ and $H'(\sigma)$ do not change, but $\|\gamma_{\hat n}(\sigma)\|$ and $\|\gamma_{\hat n}^{(1)}(\sigma)\|$ will change.  Specifically, for fixed $\sigma$, as $n > |\sigma|$ increases, the ratio $\|\gamma_n^{(1)}(\sigma)\|/\|\gamma_n(\sigma)\|$ increases as well.  So if all other terms are held constant, and $n$ is increased to $\hat n > n$, then $\alpha$ increases as well. This implies that while the \emph{theoretical} solution of \cref{eq:calcM0Normalize} will be unchanged for any $\hat n \geq n$, the \emph{numerically computed} solution of \cref{eq:calcM0Normalize} may be more accurate when  an $\hat n$ strictly greater than $n$ is used, due to better conditioning of $\bmat{\mbf U & \hat{\mbf z}(\sigma)}$.

To examine the relationship between the value of $n$ used in \cref{eq:calcM0Normalize} and the condition number of $\bmat{\mbf U & \hat{\mbf z}(\sigma)}$, we consider the same experimental set up as in \cref{sec:sysDependentAnalysis}.  We then record the condition number $\kappa_2\left(\bmat{\mbf U & \hat{\mbf z}(\sigma)}\right)$ for fixed $\sigma = e^{10^{-3}\mbf i}$ and for each value of $n$ between $100$ and $900$.  We also solve \cref{eq:calcM0Normalize} for $M_0$ to show how conditioning of $\bmat{\mbf U & \hat{\mbf z}(\sigma)}$ effects the relative error $\epsilon_{rel} \coloneqq (H(\sigma) - M_0)/H(\sigma)$.  The results of this experiment are shown in \cref{fig:CondNumvsN}.

From \cref{fig:CondNumvsN}, we see that, in general, as the value of $n$ used to construct $\mbf U$ and $\mbf z(\sigma)$ increases, the condition number of $\bmat{\mbf U & \hat{\mbf z}(\sigma)}$ decreases.  Also shown in \cref{fig:CondNumvsN} is the relative error of frequency information $M_0 \approx H(\sigma)$ recovered for the various values of $\hat n$.  The relative error shows that, for this example, using $\hat n > n$ is in fact necessary to  obtain quality frequency information.  If the true system order $n$ is used in \cref{eq:calcM0Normalize}, the corresponding relative error $\epsilon_{rel}$ is $0.90$, i.e., no meaningful information about the underlying system was recovered.  However, as $\hat n$ increases, $\epsilon_{rel}$ decreases with the condition number, down to $\epsilon_{rel} = 5.7 \times 10^{-11}$ when $\hat n = 900$.

\begin{remark}
    \label{rem:XiNotControlled}
    In the above analysis, we have omitted that changing $n$, $\sigma$, or the underlying system will all cause $\|\hat \xi_1\|$ to change, and thus the relationships that we asserted to exist between these quantities and $\alpha$ could be muddied by changes in $\|\hat \xi_1\|$.  Indeed, we do see that the relationships described above may be distorted by $\|\hat \xi_1\|$ at small scales, but over larger changes in the parameters, the relationships do tend to hold.  For example, increasing the value of $n$ used in \cref{eq:calcM0Normalize} to $\hat n = n + 1$ may not decrease the condition number, but in our numerical experiments, increasing $n$ to $\hat n = 2n$ almost always leads to to a significant decrease in the condition number.
\end{remark}

\begin{figure}[!htb]
    \centering
     \subfloat[Condition number and relative derivative]{\label{fig:CondNumVsRelDer}\includegraphics[scale = .64]{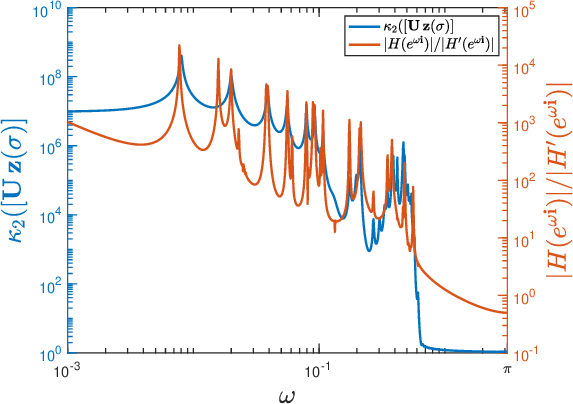}}
     \hspace{.5cm}
     \subfloat[Condition number and relative error]{\label{fig:CondNumvsN}\includegraphics[scale = .64]{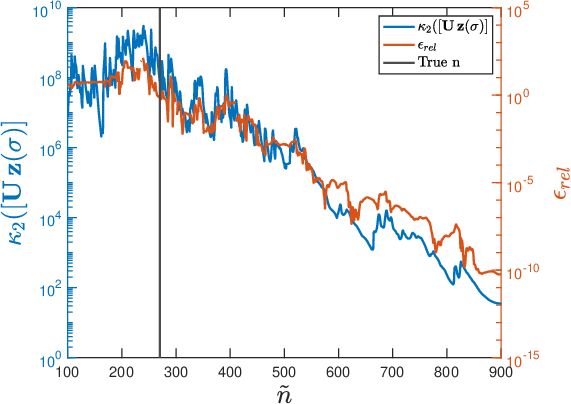}}
     \caption{Effect of relative derivative and increasing value of $n$ on condition number of \cref{eq:calcM0Normalize}} 
     \label{fig:CondNumVsRelDerAndN}
\end{figure}

\section{TD-IRKA algorithm and further considerations}
\label{sec:TD-IRKA}
TF-IRKA (\cref{alg:IRKA}) constructs locally $\mathcal H_2$ optimal reduced order models from transfer function values $H(\sigma)$ and derivatives $H'(\sigma)$.  Specifically, we require sampling the transfer function and derivative where $|\sigma| > 1$.  The analysis of the framework of \cite{burohmanBSC2020informativity} in \cite{AckermannG2023FreqROMTD} and \cref{sec:Conditioning} allow us to recover such information in inexact arithmetic for large scale systems.
Our updated method allows us to extend TF-IRKA to make use of time-domain data.  We call this extension Time-Domain-IRKA (TD-IRKA), as it constructs locally optimal $\mathcal H_2$ ROMs from a single set of time-domain input-output data ($\bb U,\bb Y$).

In this section we first discuss avoiding overflow in formation of $\mbf z(\sigma)$, an essential step for an effective implementation of TD-IRKA, then present the TD-IRKA algorithm (\cref{alg:TDIRKA}).

\subsection{Avoiding overflow}
\label{sec:AvoidOverflow}
To compute frequency data $H(\sigma)$ and $H'(\sigma)$, we must form the vector $\mbf z(\sigma)$, which has $\sigma^n$ as its last entry \cref{eq:zSigmaDef}.  We know from \cref{thm:H2optCond} that interpolation at $\sigma \in \bb C$ where $|\sigma| > 1$ is nessesary to for $\mc H_2$ optimal DDROMs.  From \cref{sec:CondNumFromSysProps}, we know that in order to accurately recover frequency information using large values of $n$ in \cref{eq:calcM0Normalize} may be necessary.  If $n$ is large, TD-IRKA may require sampling at a $\sigma$ such that $\sigma^n$ will cause overflow in double precision.

First, recall from \cref{sec:CondNumFromSysProps} that both $n$ and the relative derivative at $\sigma$, $H'(\sigma)/H(\sigma)$, play a role in the conditioning of \cref{eq:calcM0Normalize}.  Recall also that $H(z)$ is a proper rational function of $z$, and thus so is $H'(z)$.  If $H(z)$ has simple poles $\lambda_i$ (all contained in the unit disc) with residues $\phi_i$, then $H(z)$ and $H'(z)$ have the pole-residue expansions
$$H(z) = \sum_{i=1}^n\frac{\phi_i}{z-\lambda_i} \quad\quad H'(z) = \sum_{i=1}^n-\frac{\phi_i}{(z-\lambda_i)^2}.$$
So while both $|H(z)|$ and $|H'(z)|$ decay to $0$ as $|z|$ increases, $|H'(z)|$ decays faster.  Thus (assuming we are far from zeros of $H(z)$), we expect the ratio $|H'(z)|/|H(z)|$ to decay quickly as $|z|$ increases past $1$.

This observation implies that if a large $\hat n$ is required in place of $n$ in \cref{eq:calcM0Normalize} to recover $H(\sigma_1)$ for $|\sigma_1| = 1$, it is likely that some {$ \tilde n \ll \hat n$} may be sufficient to recover $H(\sigma_2)$ for $|\sigma_2| > 1$.

To test this hypothesis, we aim to recover two sets of frequency data from the same time domain input-output data, generated by the system $\mc S_1$ described in \cref{sec:motivation}.  The first set at which we will recover frequency data is $\Sigma^1 = \{\sigma_i\}_{i=1}^m$, where $m = 10^3$, $\sigma_i = e^{\mbf i \omega_i}$ and the $\omega_i$ are logarithmically distributed in $[-3,\pi)$. The second set of points to recover frequency information is $\Sigma^d = \{d\sigma_i\}_{i=1}^m$, where $d = 2.5$ and $\sigma_i$ are defined above.  The main difference between $\Sigma^1$ and $\Sigma^d$ is that each element of $\Sigma^1$ has modulus 1 and each element of $\Sigma^d$ has modulus $d = 2.5$.

\begin{figure}[!htb]
    \centering
    \includegraphics[width = .8\textwidth]{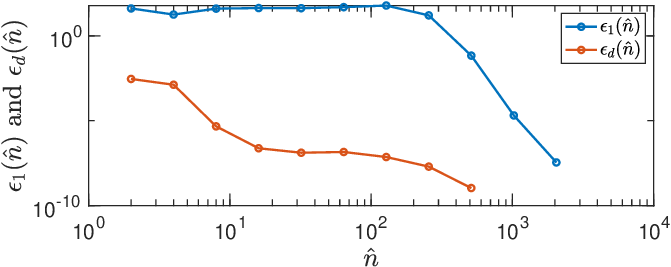}
    \caption{Errors $\epsilon_1(\hat n)$ and $\epsilon_d(\hat n)$ dependence on $\hat n$}
    \label{fig:ErrVsN_differentMod}
\end{figure}

For both sets $\Sigma^1$ and $\Sigma^d$, we recover frequency data from the same time-domain data; a Gaussian random input $\bb U \in \bb R^{T+1}$ and the corresponding output $\bb Y \in \bb R^{T+1}$, where $T = 10^4$.  We attempt to recover frequency information $H(\sigma_i)$ and $H(d\sigma_i)$ via \cref{eq:calcM0Normalize} using $\hat n = 2^1,2^2,\ldots, 2^{11}$.  To evaluate the accuracy of recovered frequency information, for each $\hat n$, we record
$$\epsilon_1(\hat n) = \max_i\frac{|H(\sigma_i)-M_0(\sigma_i)|}{|H(\sigma_i)|} \quad \text{and} \quad \epsilon_d(\hat n) = \max_i\frac{|H(d\sigma_i)-M_0(d\sigma_i)|}{|H(d\sigma_i)|}.$$
The dependence of $\epsilon_1(\hat n)$ and $\epsilon_d(\hat n)$ on $\hat n$ is displayed in \cref{fig:ErrVsN_differentMod}.

\cref{fig:ErrVsN_differentMod} clearly shows that $H(d\sigma_i)$ can be accurately recovered via \cref{eq:calcM0Normalize} with much smaller $\hat n$ than required for $H(\sigma_i)$.  These results reflect the analysis performed earlier in this section, and leads to a strategy to avoid overflow in the TD-IRKA iteration.  If TD-IRKA requires recovery of $H(\sigma)$ where $|\sigma^{\hat n}|$ causes overflow, we instead recover $H(\sigma)$ via \cref{eq:calcM0Normalize} with $\tilde n = \hat n/2$.  This strategy will  be employed in our algorithm TD-IRKA, which we discuss next.

\subsection{Algorithm}
Now we have all the pieces for the proposed method TD-IRKA. The resulting algorithm is presented in~\cref{alg:TDIRKA}.
Unlike the original TF-IRKA in~\cref{alg:IRKA} which requires re-sampling $H(z)$ and $H'(z)$ at new $r$ frequency samples at every iteration, TD-IRKA only recovers this required frequency domain information from the a single time domain input-output data $(\bb U,\bb Y)$, thus producing locally optimal $\mathcal{H}_2$ optimal DDROMs from time-domain data. In~\cref{alg:TDIRKA}, the relevant parts of the analysis is referenced in every step. 
\begin{algorithm}[!htb]
\caption{TD-IRKA}\label{alg:TDIRKA}
\begin{algorithmic}
\Require $\{\sigma_i\}_{i = 1}^r$, an initial set of interpolation points closed under conjugation, time domain data $(\bb U,\bb Y)$, and working system order $\hat n$.
\State Recover $M_0 \approx H(\sigma_i)$ and $M_1 \approx H'(\sigma_i)$ at $\{\sigma_i\}_{i = 1}^r$ via the process described in \cref{sec:Conditioning} and \cite{AckermannG2023FreqROMTD}.
\State Form $\mbf E_{r}$ and $\mbf A_{r}$ via Hermite Loewner interpolation (see \cref{sec:HerLowInterp}, \cite{MayoA2007LoewnerFramework,AntoulasBG2020Book})
\While{Not Converged}
\State Calculate generalized eigenvalues $\{\lambda_i\}_{i=1}^r$ of the pencil $\lambda\mbf E_{r} - \mbf A_{r}$
\State Set $\sigma_i \gets \frac{1}{\lambda_i}$
\State Recover $M_0 \approx H(\sigma_i)$ and $M_1 \approx H'(\sigma_i)$ at $\{\sigma_i\}_{i = 1}^r$ via the process described in \cref{sec:Conditioning} and, if required, the overflow procedure as discussed in \cref{sec:AvoidOverflow}
\State Form new $\mbf E_{r}$ and $\mbf A_{r}$ via Hermite Loewner interpolation.
\EndWhile
\State Form final $\mbf E_{r},\mbf A_{r},\mbf b_{r},\mbf c_{r}$ via Hermite Loewner interpolation \cref{sec:HerLowInterp}.
\end{algorithmic}
\end{algorithm}

\section{Numerical Examples}
\label{sec:results}
We now display TD-IRKA's ability to construct $\mathcal H_2$ optimal DDROMs from time-domain data $(\bb U,\bb Y)$.  We will compare the quality of DDROMs produced by TD-IRKA (using only time-domain data $(\bb U,\bb Y)$ to those produced by TF-IRKA with the same initial interpolation set $\{\sigma_i\}_{i=1}^r$ and functions to compute the true values of $H(\sigma_i)$ and $H'(\sigma_i)$.  These numerical experiments were carried out on a 2023 MacBook Pro equipped with 16 GB RAM and an Apple M2 Pro chip running macOS Sonoma 14.4.1.  All algorithms are implemented in \MATLAB version 24.1.0.2578822 (R2024a) Update 2.  All code and data are available in \cite{supAck24}.

In all examples, we first construct an initial order $r$ reduced model $\widehat H_{\texttt{VF}}$ via vector fitting \cite{DeschrijverMDZ2008FastVF,Gustavsen2006PoleRelocVF,semlyenG1999VF}. Given $m$ data samples $H(\sigma_i),i=1,\ldots,m$ where $m > r$, vector fitting seeks to minimize the least squares error $\sum_{i = 1}^m\left|H(\sigma_i) - \widehat H_{\texttt{VF}}(\sigma_i)\right|^2$.  We then take the initial interpolation set to be the reciprocal of the poles of $H_{\texttt{VF}}$.  This idea was proposed in \cite{DramacGB2015QuadVF} as a preprocessing step to speed up convergence of TF-IRKA.  In all examples, the initial Vector Fitting model is constructed with $m = 3r$ values of the transfer function $H(e^{\omega_i\mbf i})$, where the $\omega_i$ are logarithmically spaced in $(0,\pi)$ and $r$ is the reduced order.  For TD-IRKA, the values $H(e^{\omega_i\mbf i})$ are obtained via the process described in \cref{sec:preliminaries} and \cite{AckermannG2023FreqROMTD}, and for TF-IRKA the values are obtained from the true transfer function.

\begin{remark}
    \cref{thm:H2optCond} gives simple criteria to check that the models recovered by TD- and TF-IRKA are indeed locally $\mathcal H_2$ optimal.  We cannot expect to satisfy these criteria exactly as our recovered frequency data will not be exact.  Further, the influence of stopping criterion for IRKA iterations is still an open question, see \cite{BeattieDG2019IRKApolePlace}, and the use of data-driven Hermite Loewner interpolatory models adds additional complications as these models can suffer from numerical instabilities, see \cite{embreeI2019LowPseudoSpec} for details.  Due to these considerations, we do not expect our models to exactly satisfy \cref{thm:H2optCond}.  Nevertheless, the results below clearly show that TD-IRKA is capable of producing near-optimal models from time-domain data that have similar approximation power to those created with TF-IRKA.
\end{remark}

\subsection{Linear advection model}
\label{sec:advectionModel}
We first consider a model of linear advection obtained from \cite{morwiki_linad}.  The model is a finite element discretization of the one-dimensional linear advection equation
\begin{equation}
    \label{eq:advectionEq}
    \frac{\partial v}{\partial t} = -a\frac{\partial v}{\partial x},
\end{equation}
where the left boundary $v(0,t)$ is the input $U(t)$ and the right boundary $v(1,t)$ is the output $Y(t)$.  We take the transport velocity $a$ to be $20$. We use the function provided in \cite{morwiki_linad} to create an order 
$n = 10^3$ finite element discretization of \cref{eq:advectionEq}, a linear continuous-time dynamical system.  We then discretize this model in time using sampling frequency $10^4$ Hz and a zero-order hold discretization method (\MATLAB's \texttt{c2d} command), which gives us the model $\mc S_2$ of the form \cref{eq:LinDiscSys} with associated transfer function $H_2(z)$.

To test the performance of \cref{alg:TDIRKA}, we first simulate $\mc S_2$ with a Gaussian random input $\bb U \in \bb R^{10^4 + 1}$ to obtain the corresponding output $\bb Y \in \bb R^{10^4 + 1}$.  For this data, we find that using $\hat n = 600$ in \cref{alg:TDIRKA} results in sufficiently accurate recovered frequency information, as determined by the error indicator from \cite{AckermannG2023FreqROMTD}.

Using $\bb U$ and $\bb Y$, we form DDROMs of $\mc S_2$ using TD-IRKA (\cref{alg:TDIRKA}), $\widehat H_{\texttt{TD}}$, for orders $r = 4,6,\ldots, 30$.
For comparison, we also form DDROMs $\widehat H_{\texttt{TF}}$ of $\mc S_2$ using TF-IRKA (\cref{alg:IRKA}), which was provided with $H_2(z)$ and its derivative $H'_2(z)$. To compare performance, we record the relative $\mc H_2$ error
\begin{equation}
    \label{eq:relH2err}
    \|H_2-\widehat H_{\texttt X}\|_{\mc H_2}/\|H_2\|_{\mc H_2},
\end{equation}
where $\texttt X$ is either $\texttt{TF}$ or $\texttt{TD}$.
These $\mc H_2$ errors are displayed in \cref{fig:Advection_H2Converge}.

\begin{figure}[!htb]
    \centering
    \includegraphics[width=.8\textwidth]{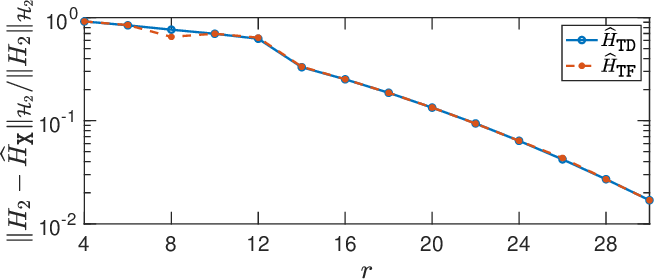}
    \caption{Relative $\mc H_2$ errors of TF-IRKA and TD-IRKA DDROMs approximating $H_2$.}
    \label{fig:Advection_H2Converge}
\end{figure}

We clearly see from \cref{fig:Advection_H2Converge} that \cref{alg:TDIRKA} is able to produce DDROMs that approximate $\mc S_2$ with nearly the same accuracy as \cref{alg:IRKA}.  The number of iterations required for convergence was comparable between the algorithms. For example, for $r=30$ TF-IRKA converged in 
35 steps and thus required
resampling the true functions  $H(z)$  and $H'(z)$ at more than $1,050$ frequencies. This is what true $\mathcal{H}_2$ optimality requires. But now with TD-IRKA, with a single time-domain input-output trajectory, we are able to capture TF-IRKA accuracy without ever sampling either $H(z)$ or $H'(z)$.

\subsection{ISS model}
We next consider the ISS1R model ($\mc S_1$ with transfer function $H_1(z)$), described in \cref{sec:motivation}).  We simulate $\mc S_1$ for $T = 10^4$ time steps with Gaussian random input $\bb U \in \bb R^{T+1}$ to obtain the output $\bb Y \in \bb R^{T+1}$.  Then, for each $r = 2,4,\ldots,30$, we construct order $r$ locally $\mc H_2$ optimal reduced order models with TF-IRKA ($\widehat H_{\texttt{TF}},   $ \cref{alg:IRKA}) and TD-IRKA ($\widehat H_{\texttt{TD}}$, \cref{alg:TDIRKA}).  The TD-IRKA models are all constructed from the same $\bb U$, $\bb Y$ pair; only a single time domain simulation was performed.  To recover the frequency information to create the TD-IRKA models, we use $\hat n = 900$ (see \cref{sec:CondNumFromSysProps}).  We record the relative $\mc H_2$ errors (\cref{eq:relH2err}) in \cref{fig:ISS_H2_converge}.

\begin{figure}[!htb]
    \centering
    \includegraphics[width=.8\textwidth]{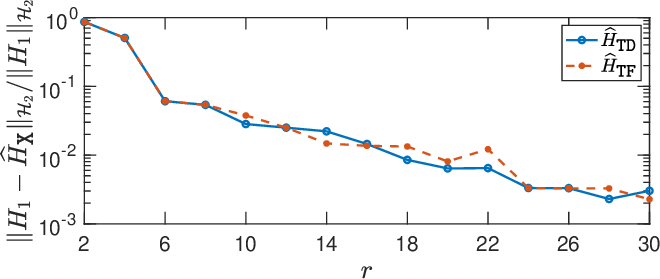}
    \caption{Relative $\mc H_2$ errors of TF-IRKA and TD-IRKA DDROMs approximating $H_1$.}
    \label{fig:ISS_H2_converge}
\end{figure}

\cref{fig:ISS_H2_converge} once again confirms that TD-IRKA tracks the performance of TF-IRKA as the order of the reduced model increases.  In contrast to the advection example (\cref{sec:advectionModel}), we observe that for several values of $r$ there is a visual difference between the $\mc H_2$ performance of TF-IRKA and TD-IRKA.  In each case, both methods are initialized with the same initial interpolation set and both methods converged to a locally $\mc H_2$ optimal DDROM.  However, due to the lower accuracy of the frequency information recovered from time-domain data $(\bb U, \bb Y)$, TD-IRKA converged to a different local minimum, {indeed with a smaller $\mathcal{H}_2$ error in many cases}.  The reason of this behavior is likely that small errors in recovered frequency information can lead to changes in the recovered eigenvalues of the Loewner matrix pencil $z\mbf L - \mbf M$ (\cref{sec:HerLowInterp}) \cite{embreeI2019LowPseudoSpec}.  The consideration and mitigation of such effects are reserved for future research.

We also remark that for this example, the technique suggested to avoid overflow in the calculation of $\sigma^{\hat n}$ discussed in \cref{sec:AvoidOverflow} was (automatically) used several times.  Clearly, this technique did not prevent TD-IRKA from converging.

\subsection{Heat equation}

We now consider a model of 1-dimensional heat flow in a rod.  The left boundary of the rod is kept at constant temperature $T(0,t) = 0$, and at the right boundary, the rate of heat flow into the rod is given by the input $U(t)$.  The temperature $T(x,t)$ is modeled via the PDE 
\begin{equation}
    \label{eq:HeatPDE}
    C_p\rho \frac{\partial T(x,t)}{\partial t} = K_0\frac{\partial^2 T(x,t)}{\partial x^2}; \quad t > 0; \quad x \in (0,1),
\end{equation}
together with the boundary conditions
\begin{center}
  \begin{subequations}
    \parbox[c]{\textwidth*14/30}{
    \begin{equation}
        \label{eq:LeftBC}
        T(0,t) = 0
    \end{equation}
    }~\hspace{0.19in}
    \parbox[c]{\textwidth*14/30}{
    \begin{equation}
        \label{eq:RightBC}
        \frac{\partial T(1,t)}{\partial t} = U(t).
    \end{equation}
    }~\hspace{0.19in}
  \end{subequations}
\end{center}
These boundary conditions cause \cref{eq:HeatPDE} to be asymptotically stable \cite{CurtainM2009TransferFuns}. We choose the specific heat capacity as $C_p = 0.896$, mass density as $\rho = 2700$, and thermal conductivity as $K_0 = 167$, which correspond to the thermal properties of 6061 Aluminum at 300 Kelvin \cite{LeinhardL2008HeatTransfer}.  The output $Y$ is taken to be the temperature at $x = 0.8$.  The input $U(t)$ is the zero-order hold of a discrete random vector $\bb U \in \bb R^{4001}$.

We simulate the response $Y(t)$ of the PDE \cref{eq:HeatPDE} to the input $U(t)$ for 4 seconds with sampling frequency $\omega_s = 10^{-3}$ (so $\bb U \in \bb R^{2/\omega_s + 1}$).  The simulation was performed with \MATLAB's \texttt{pdepe} PDE solver.   The vector of discrete outputs is denoted $\bb Y \in \bb R^{4/\omega_s + 1}$.  We use this $\bb U$ and $\bb Y$ to construct a TD-IRKA model of order $r = 4$.

We emphasize that in contrast to the previous two examples, where we had an explicit discrete-time dynamical system realization as in \cref{eq:LinDiscSys}, in this example we directly simulate the underlying PDE \cref{eq:HeatPDE}. We do not assume access to the transfer function of an underlying fully-discretized model. We only assume access to the output of a black-box simulation as is the case in many practical situations. \emph{Thus, one cannot produce DDROMs with 
TF-IRKA in this case since we do not have access to $H(z)$}. 

To assess the quality of our reduced model, we simulate both the DDROM and the PDE for a test input $\bb U_{\text{test}}$ chosen as the zero order hold of a sawtooth wave with frequency 10 Hz and amplitude of 10 for 4 seconds, thus $\bb U_{\text{test}} \in \bb R^{4/\omega_s + 1}$.  The response of the DDROM to $\bb U_{\text{test}}$ is $\bb Y_{\text{TD}} \in \bb R^{4/\omega_s + 1}$ and the response of the PDE to $\bb U_{\text{test}}$ is $\bb Y_{\text{test}} \in \bb R^{4/\omega_s + 1}$.  Both outputs are displayed in \cref{fig:HeatROMTraj} and the corresponding error $|\bb Y_{\text{test}}-\bb Y_{TD}|$ is displayed in \cref{fig:HeatROMError}.

\begin{figure}[!htb]
    \centering
     \subfloat[Trajectories]{\label{fig:HeatROMTraj}\includegraphics[scale = .64]{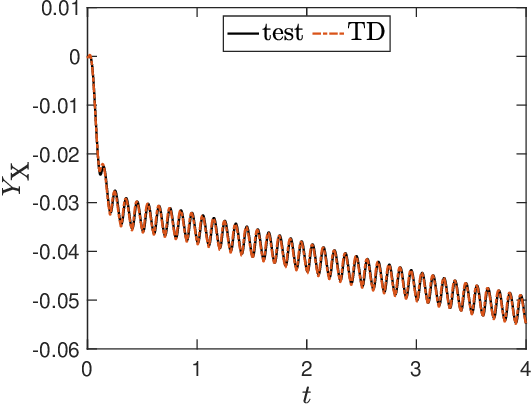}}~~~~
     \subfloat[Error]{\label{fig:HeatROMError}\includegraphics[scale = .64]{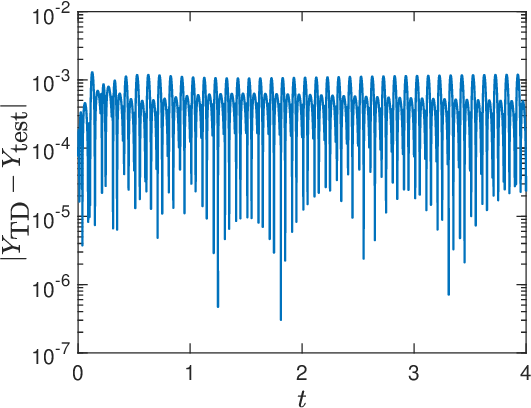}}
     \caption{Trajectories $\bb Y_{\text{test}}$ and $\bb Y_{\text{TD}}$ of \cref{eq:HeatPDE} and an order $r = 4$ DDROM obtained via \cref{alg:TDIRKA} when driven by $\bb U_{\text{test}}$, and associated error $|\bb Y_{\text{test}}-\bb Y_{\text{TD}}|$.}
     \label{fig:TrajectoryAndErrorHeat}
\end{figure}

\cref{fig:TrajectoryAndErrorHeat} shows that 
the TD-IRKA DDROM provides a high-fidelity approximation to the solution obtained from the PDE solver; the maximum error was $1.29 \times 10^{-3}$.  These results indicate \cref{alg:TDIRKA} produces 
high-fidelity (in this case locally $\mathcal{H}_2$ optimal) DDROMs 
using only time-domain black-box simulations.  These results are especially encouraging since even though the data-informativity theory assumes an underlying rational function to approximate, this example seems to illustrate that the method can be applied to a broader class of functions.

\section{Conclusions}
We presented an analysis of the data informativity framework for interpolation in the special case of recovering frequency information at points outside the unit disc.  This analysis included an optimal scaling result for conditioning of nearly unitary matrices, as well as connecting conditioning of the linear systems arising from the data informativity framework to properties of the underlying dynamical system.  These considerations were key in producing a variant of IRKA, TD-IRKA, which can produce a DDROM of a dynamical system provided with only a single time-domain simulation of the dynamical system.  TD-IRKA was demonstrated to match the convergence behavior of TF-IRKA for two benchmark systems and was also shown to work on data obtained from a PDE solver, without first forming a linear discrete-time dynamical system. Extending the data-informativity framework for interpolation to special classes of nonlinear systems, such as bilinear systems, could be an interesting research direction to pursue to further this framework.

\bibliographystyle{siamplain}
\bibliography{references2}

\end{document}